\documentclass[10pt,reqno]{amsart}

\usepackage{amsmath, amsfonts, amssymb, mathrsfs}
\usepackage{hyperref}
\usepackage{cleveref}

\usepackage[noadjust]{cite}
\usepackage{graphicx}
\usepackage[caption=false]{subfig}
\usepackage{algorithmic}
\usepackage[ruled,vlined]{algorithm2e}

\def\E{ \mathrm{e} }									% e
\def\T{ \mathrm{T} }									% T - transpose
\def\NN{ \mathbb{N} }									% positive integers
\def\QN{ \mathbb{Q} }									% rational numbers
\def\RN{ \mathbb{R} }									% real numbers
\def\RNc{ \mathbb{R}_{\mathrm{c}} }		% computable real numbers
\def\CN{ \mathbb{C} }									% complex numbers
\def\C{ \mathcal{C} }									% Continuous Functions
\def\Cc{ \mathcal{C}_{\mathrm{c}} }		% computable continuous functions
\def\A{ \mathcal{A} }									% Set of Assignment functions
\def\R{ \mathcal{R} }									% Rectangle
\def\M{ \mathcal{MIN} }								% Set of Minimizers

\def\bsx{ \boldsymbol{x} }						% Sequence x
\def\bsxi{ \boldsymbol{\xi} }					% Sequence xi

\def\bx{ \mathbf{x} }								% bold x
\def\by{ \mathbf{y} }								% bold y
\def\TM{ \mathrm{TM} }              % Turing machine
\def\Min{ \mathrm{Min} }

\newcommand{\Op}[1]{\mathrm{#1}}								% general Operators

\theoremstyle{plain}
\newtheorem{theorem}{Theorem}[section]
\newtheorem{proposition}[theorem]{Proposition}
\newtheorem{lemma}[theorem]{Lemma}
\newtheorem{corollary}[theorem]{Corollary}

\theoremstyle{definition}
\newtheorem{definition}{Definition}

\theoremstyle{remark}
\newtheorem{remark}{Remark}
\newtheorem{example}{Example}
\newtheorem{question}{Question}
\begin{document}

% ========== Fontmatter ==========
\title[Iterative Optimization on Turing Machines]{Iterative Optimization of Multidimensional Functions on Turing Machines under Performance Guarantees}
\thanks{This work was partly supported by the German Federal Ministry of Education and Research (BMBF) within the national initiative on 6G Communication Systems through the research hub 6G-life under Grant 16KISK002, 
and within the national initiative for Quantum Communication in the BMBF quantum programs QD-CamNetz, Grant 16KISQ077, QuaPhySI, Grant 16KIS1598K, QUIET, Grant 16KISQ093, and QR-X, Grant 16KISQ037K.
The work of H.~V.~Poor was partly supported by the U.S. National Science Foundation under Grand~CNS-2128448.}

\author[H.~Boche]{Holger Boche}
\address{Holger Boche\\
				Technische Universit{\"a}t M{\"u}nchen,
				Lehrstuhl f{\"u}r Theoretische Informationstechnik,
				Arcisstrasse 21, 80290 M{\"u}nchen, Germany
         %%Tel.: +49 89 289 23240\\
         %%Fax.: +49 89 289 23242
				}
\email{boche@tum.de}

\author[V.~Pohl]{Volker Pohl}
\address{Volker Pohl\\
        Technische Universit{\"a}t M{\"u}nchen,
				Lehrstuhl f{\"u}r Theoretische Informationstechnik,
				Arcisstrasse 21, 80290 M{\"u}nchen, Germany
         %Tel.: +49 89 289 23250\\
         %Fax.: +49 89 289 23242
				}
\email{volker.pohl@tum.de}

\author[H. V. Poor]{H. Vincent Poor}
\address{H. Vincent Poor\\
				Princeton University,\\
        Department of Electrical and Computer Engineering,\\	    
				Princeton, NJ 08544, USA
         %Tel.: +49 89 289 23250\\
         %Fax.: +49 89 289 23242
				}
\email{poor@princeton.edu}

%\author{Holger Boche\thanks{Chair of Theoretical Information Technology, Technical University of Munich, Arcisstra{\ss}e 21, 80333 M{\"u}nchen, Germany (\email{boche@tum.de}, \email{volker.pohl@tum.de}).}
%\and Volker Pohl\footnotemark[2]
%\and H. Vincent Poor\thanks{Department of Electrical and Computer Engineering, Princeton University, Princeton, NJ 08544, USA (\email{poor@princeton.edu}).}}

%\headers{Iterative Optimization on Turing Machines}{Holger Boche, Volker Pohl, and H. Vincent Poor}
% ================================

\subjclass[2000]{Primary 90C25, 90C06; Secondary 68W40, 03D80}

\date{\today}

% ==============================
% ========== Abstract ==========
% ==============================
\begin{abstract}
This paper studies the effective convergence of iterative methods for solving convex minimization problems using block Gauss--Seidel algorithms. 
It investigates whether it is always possible to algorithmically terminate the iteration in such a way that the outcome of the iterative algorithm satisfies any predefined error bound.
It is shown that the answer is generally negative. 
Specifically, it is shown that even if a computable continuous function which is convex in each variable possesses computable minimizers, a block Gauss--Seidel iterative method might not be able to effectively compute any of these minimizers. 
This means that it is impossible to algorithmically terminate the iteration such that a given performance guarantee is satisfied.
The paper discusses two reasons for this behavior.
First, it might happen that certain steps in the Gauss--Seidel iteration cannot be effectively implemented on a digital computer.
Second, all computable minimizers of the problem may not be reachable by the Gauss--Seidel method.
Simple and concrete examples for both behaviors are provided.
\end{abstract}

\keywords{
Convex optimization,
decentralized optimization,
computability,
Gauss--Seidel method,
Turing machine
}

%\begin{MSCcodes}
%90C25, 90C06, 68W40, 03D80
%\end{MSCcodes}

\maketitle
% ==================================
% ========== INTRODUCTION ==========
% ==================================
\section{Introduction}
\label{sec:Indro}

Many problems in physics and engineering can be formulated as optimization problems.
The most important example may be the second law of thermodynamics which may be restated as the principle of minimum energy.
But also problems of finding the optimal allocation of limited resources, 
signal processing and signal recovery problems \cite{PalomarEldar_ConvexOpt,liu2024survey},
solutions for compressed sampling problems \cite{CandesTao_IT05,CandesRecht_MatrixCompl09,DeSantis_SIAMJOpt16},
or problems in financial mathematics \cite{BenidisFengPalomar_Now18} and operations research can be formulated as optimization problems.
In particular, a large number of problems and challenges in artificial intelligence and data science are formulated as optimization problems \cite{LeeBoKu_CompOpt_24}. 
Modern computer technology makes these theoretical optimization problems extremely powerful engineering tools because
fast digital hardware allows one to solve even huge optimization problems in high dimensions very fast based on advanced algorithms developed over the past few decades for the different optimization problems.

Nevertheless, the question of whether a particular value "solves" a certain problem depends on the actual requirements on the "solution" by the user.
Already immediately after the introduction of the mathematical notion of \emph{computing}, due to Turing \cite{Turing_1937,Turing_1938}, it became clear that few physical problems can be solved exactly by digital computations.
For most problems, the exact solution can only be approximated (arbitrarily well) by the outcome of computations on a digital computer. 
For this reason, Turing required that the result of the computing process satisfy a predefined bound on the approximation error.
We will discuss this requirement later in more detail, because the underlying question of this paper is whether it is always possible to control this approximation error for solutions obtained by iterative optimization methods.
At this point, we only mention that the algorithms from \cite{PalomarEldar_ConvexOpt,liu2024survey,CandesTao_IT05,CandesRecht_MatrixCompl09,BenidisFengPalomar_Now18} as many other algorithms in applications, do not satisfy this condition, i.e. in all these algorithms it is not possible to control the difference between the computed "solution" and the true value of the problem.
The mentioned limitation of digital computation mainly originates from the fact that a digital computer can compute exactly only with rational numbers.
Real numbers can generally not be represented exactly on a digital computer, but rather can only be approximated by rational numbers.
If a real number can effectively (i.e. by controlling the approximation error) be approximated by rational numbers, it is said to be \emph{(Turing-) computable}  (cf. \Cref{sec:Notation} for details). Otherwise it is said to be \emph{non-computable}.

This paper studies some consequences of this limitation of digital computers on the ability to solve optimization problems on digital hardware.
We basically consider the following simple optimization problem in the $m$-dimensional real Euclidean space $\RN^{m}$:
\begin{equation}
\label{equ:MinProbGen}
	\min_{\bx\in\R} f(\bx) 
\end{equation}
where $f : \RN^{m} \to \RN$ is a continuous function and $\R \subset \RN^{m}$ is a convex and compact subset of $\RN^{m}$.
It is well known \cite{PourEl_Computability} that the minimum value $\Op{Min}_{\R}(f) = \min_{\bx\in\R} f(\bx)$ of this optimization problem is always Turing computable.
Moreover, if the problem has a unique minimizer, i.e. if the set 
\begin{equation}
\label{equ:SetGlobalMinimizer}
	\M_{\R}(f) = \big\{ \widehat{\bx}\in\R : f(\widehat{\bx}) = \Op{Min}_{\R}(f) \big\}
\end{equation}
contains exactly one vector, then also this minimizer is Turing computable.
If the minimizer is not unique, then some or all of the minimizers might not be Turing computable \cite{PourEl_Computability} which means that these optimizers cannot effectively be computed on any digital computer. 

Still, as long as the optimization problem contains at least one computable minimizer, this minimizer can, in principle, algorithmically be computed on a digital computer.
One only needs to construct a sequence $\left\{ \bx^{(k)} \right\}_{k\in\NN}$ of computable vectors $\bx^{(k)}\in\RN^{m}$ that effectively converges to the computable minimizer $\widehat{\bx}$.
This is always possible, since the minimizer is computable, although finding such a sequence might be a fairly complicated and creative problem. 
The question is now whether this process, of finding a computable sequence that effectively converges to a computable minimizer, can (efficiently) be automated on a digital computer \cite{ComputingAsDisciplin}.

Practical algorithms usually apply a certain (suboptimal) strategy to find such a sequence $\big\{ \bx^{(k)} \big\}_{k\in\NN}$ that converges to a minimizer of the given function.
A very popular and successful strategy is to minimize successively only over one (or several) coordinates of $\bx\in\RN^{m}$ while keeping the other dimensions fixed.
This approach appears in the literature in many different variants and under various names such as block Gauss-Seidel methods, block coordinate decent method, or block coordinate update method, to mention only some of many variants of the general idea \cite{Beck_SIAMJOpt13,Wright_MathProg15,DeSantis_SIAMJOpt16,GrippoGaussSeidel2000,Razaviyayn_SIAMJopt13,Xu_SIALJOpt}.
Thus, if
\begin{equation}
\label{equ:Rectangle1}
	\R = [a_{1},b_{1}] \times [a_{2},b_{2}] \times \dots \times [a_{m},b_{m}]
\end{equation}
and if $\bx^{(0)} = (x^{(0)}_{1}, x^{(0)}_{2}, \dots, x^{(0)}_{m}) \in \RN^{m}$ is an arbitrary initialization vector then one successively solves, for $\ell=1,2,\dots,m$, the one dimensional optimization problems
\begin{equation}
\label{equ:ItarativLocalOpt_1D}
	x_{\ell}^{(k+1)} = \arg \min_{y\in [a_{\ell},b_{\ell}]} f(x_{1}^{(k+1)}, \dots, x_{\ell-1}^{(k+1)}, y, x_{\ell+1}^{(k)}, \dots, x_{m}^{(k)})	
\end{equation}
and iterates then over $k=0,1,2, \dots$.
This strategy produces a sequence $\left\{ \bx^{(k)} \right\}_{k\in\NN}$ that converges to a global minimizer $\widehat{\bx}$ of $f$ under fairly weak (convexity) conditions on $f$.
However, the so determined sequence $\left\{ \bx^{(k)} \right\}_{k\in\NN}$ may not \emph{effectively} converge to the global minimizer.
Instead, it may happen that:
\begin{itemize}
\item Some of the computational steps of the algorithm are not Turing computable, i.e. they may not be realizable on a digital computer with an effective control of the approximation error.
\item The sequence $\big\{ \bx^{(k)} \big\}_{k\in\NN}$ may converge to a non-computable minimizer of $f$.
\end{itemize}
Thus even through there exists a global computable optimizer of $f$, the automated procedure might not be able to find a corresponding approximation sequence $\left\{ \bx^{(k)} \right\}_{k\in\NN}$ of computable vectors that effectively converges to this minimizer.
So the automated procedure, i.e. the algorithm, might not be able to compute a global minimizer of $f$.

In this paper, we study the described iterative optimization strategy for finding the global minimizer of a function $f$.
We will show that there exists a very simple (piecewise linear) function $f$ that is convex in each coordinate (but not jointly convex),
that has infinitely many computable minimizers but such that the iterative optimization method is not able to find a computable approximation sequence that converges effectively to any of the computable minimizers.
In fact, we will give an example of a function $f$ for which some of the computational steps of the iterative algorithm are not Turing computable,
and we will provide an example of a function $f$, for which the iterative algorithm will always converge to non-computable minimizers, even though $f$ possesses also computable minimizers.

The remainder of this paper is organized as follows.
\Cref{sec:Notation} introduces our main notation and recalls some concepts from computability analysis that will be needed in the paper.
\Cref{sec:OptimizationProblem} formulates and states the optimization problem under consideration and it explains the iterative algorithm that will be investigated in greater detail.
\Cref{sec:ProblemWithAssignmentFunc,sec:Reachability} will then provide examples that illustrate that the convergence of the iterative optimization algorithm is generally not effective.
The paper closes with a short discussion on possible extensions in \Cref{sec:Extension} and with summary in \Cref{sec:Summary}.

% ================================================================================================================
% ========== PRELIMINARIES =======================================================================================
% ================================================================================================================
\section{Notation and preliminaries}
\label{sec:Notation}

We write $\RN$ for the set of real numbers, and $\RN_{+}$ and $\RN_{-}$ for the subset of all positive and negative real numbers, respectively.
Throughout this paper, we consider functions defined on the usual $m$-dimensional Euclidean space $\RN^{m}$ for some dimension $m\geq 1$.
Vectors in $\RN^{m}$ will be denoted by boldface lower-case letters and they will be written as row vectors like $\bx = (x_{1}, x_{2}, \dots, x_{m})$.
As usual, $[a,b] = \left\{ x\in \RN : a\leq x \leq b\right\}$ denotes a closed interval on $\RN$. 
For real numbers $a_{i} < b_{i}$, $i=1, 2, \dots, m$, the Cartesian product \eqref{equ:Rectangle1} is said to be a (closed) \emph{rectangle} in $\RN^{m}$.
We say that $\R$ is a \emph{computable rectangle}, if all $a_{i}, b_{i}$, $i=1, 2, \dots, m$ are computable numbers (see \Cref{def:CompNumber} below). 
The set of all continuous functions defined on $\RN^{m}$ or on a rectangle $\R \subset \RN^{m}$ are denoted by $\C(\RN^{m})$ or $\C(\R)$, respectively.
Similarly, for any $K\in\NN$, $\C^{K}(\RN^{m})$ and $\C^{K}(\R)$ denotes the set of all $K$-times continuously differentiable functions on $\RN^{m}$ and $\R$, respectively.

This paper investigates the optimization problem~\eqref{equ:MinProbGen} for at least piecewise differentiable functions $f$.
A point $\widetilde{\bx} \in \R$ is said to be a \emph{critical point} for problem~\eqref{equ:MinProbGen} if
\begin{equation}
\label{equ:defCriticalPoint}
	\nabla f(\widetilde{\bx}) (\by - \widetilde{\bx})^{\T} \geq 0\,,
	\quad\text{for all}\ \by\in \R,
\end{equation}
where $\nabla f(\widetilde{\bx}) = (\partial f/\partial x_{1}, \partial f/\partial x_{2}, \dots, \partial f/\partial x_{m})(\widetilde{\bx})$ denotes the gradient of $f$ at $\widetilde{\bx}$.

% ===== Computability analysis =====
\subsection{Computability analysis}
\label{sec:CompAnalysis}

This section briefly reviews the main concepts an notion of computability analysis as far as they are needed in this paper. 
We refer to standard textbooks (e.g., \cite{Turing_1937,Turing_1938,PourEl_Computability,Weihrauch_ComputableAnalysis,Friedman_CompComplexity_84,Ko_Complexity91}) for more detailed expositions.

The central concept of computability analysis is the notion of effective convergence.

\begin{definition}[Effective convergence]
Let $\bsx = \left\{ x_{n} \right\}_{n\in\NN}$ be a sequence of real numbers that converges to $x\in\RN$.
We say that $\bsx$ converge effectively to $x$ if
\begin{equation*}
	\left|x - x_{n}\right| \leq 2^{-n}\,,
	\quad\text{for all}\ n\in\NN\,.
\end{equation*}
\end{definition}
So for a sequence that effectively converges, it is possible to control the approximation error $\left|x - x_{n}\right|$,
in the sense that for any arbitrary small approximation error $\epsilon = 2^{-n}$ 
it is possible to determine algorithmically an index $n\in\NN$ such that the approximation error is guaranteed to be less than $\epsilon$.

Every $x\in\RN$ is the limit of a sequence of rational numbers, but only if there exists a rational sequence that effectively converges to $x$, it is said to be computable.

\begin{definition}[Computable number and vector]
\label{def:CompNumber}
An $x\in\RN$ is said to be \emph{computable} if there exists a sequence $\left\{ r_{k} \right\}_{k\in\NN} \subset\QN$ of rational numbers that converges effectively to $x$.
In this case, the sequence $\left\{ r_{k} \right\}_{k\in\NN}$ is said to be a \emph{representation} of $x$.\\
A vector $\bx\in\RN^{m}$ is said to be computable if each of its components is a computable number.
\end{definition}
Subsequently, we write $\RNc \subsetneq \RN$ for the proper subfield of all computable real numbers and $\RNc^{m}$ for the set of all computable vectors in $\RN^{m}$.

\begin{definition}[Computable sequence]
A sequence $\bsx = \left\{ x_{n} \right\}_{n\in\NN}$ of real numbers is said to be \emph{computable} if there exists a double index sequence $\left\{ r_{n,k} \right\}_{n,k\in\NN} \subset\QN$ of rational numbers such that for every $n\in\NN$
\begin{equation*}
	\left|x_{n} - r_{n,k} \right| < 2^{-k}\,,
	\quad\text{for all}\ k\in\NN\,.
\end{equation*}
A sequence $\left\{ \bx_{n} \right\}_{n\in\NN} \subset\RN^{m}$ is said to be computable if every component is a computable sequence.
\end{definition}

Besides computable numbers and sequences, \emph{computable functions} will play an important role in this paper. 
There are different notions of computable functions and we refer to \cite{AvigadBrattka_2014} for an overview of these different notions.
Here, we need in particular the following two:

\begin{definition}[Computable function]
Let $m,M\in\NN$. A function $f : \RN^{m} \to\RN^{M}$ is said to be
\begin{itemize}
\item \emph{Borel--Turing computable},
if there exists a Turing machine that transforms every representation $\left\{ \bx_{n} \right\}_{n\in\NN}$ of $\bx$ to a representation of $f(\bx)$.
\item \emph{Banach--Mazur computable},
if for every computable sequences $\left\{ \bx_{n} \right\}_{n\in\NN} \subset \RN^{m}$ the sequence $\left\{ f(\bx_{n}) \right\}_{n\in \NN} \subset\RN^{M}$ is computable.
\end{itemize}
\end{definition}
We note that the notion of being Banach--Mazur computable is more general than being Borel--Turing computable.
Thus every Borel--Turing computable function is also Banach--Mazur computable but there exist Banach--Mazur computable functions that are not Borel--Turing computable.

A well known example of a function that is not Banach--Mazur computable is the sign function.
This example will be of some importance in this paper. 
Therefore we state this result here in the form that is needed later.

\begin{lemma}
\label{lem:Galpha}
Let $a\in\RNc$, $a>0$, let $\alpha \in [-1,1]$ be arbitrary, and let $G_{\alpha} : [-a,a] \to \RN$ be the function defined by
\begin{equation*}
	G_{\alpha}(x) = \left\{\begin{array}{rll}
	1 & : & x<0\\
	\alpha & : & x=0\\
	-1 & : & x>0
	\end{array}\right. .
\end{equation*}
Then $G_{\alpha}$ is not Banach--Mazur computable and therefore not Borel--Turing computable.
\end{lemma}

\begin{proof}[Sketch of proof]
Assume $\alpha \in [-1,1]$ is not a computable number. Then $G_{\alpha}(0)\notin\RNc$ and so $G_{\alpha}$ cannot be Banach--Mazur computable.
If $\alpha\in [-1,1]$ is a computable number, then we can find a computable sequence of computable numbers $\left\{ x_{n} \right\}_{n\in\NN}$ such that 
$\left\{ G_{\alpha}(x_{n})\right\}_{n\in\NN}$ is not a computable sequence of computable numbers. 
To construct such a sequence $\left\{ x_{n} \right\}_{n\in\NN} \subset \RNc$, we can use a technique from the proof of \cite[Theorem~$2$]{BSP_TIT23}.
\end{proof}

\begin{definition}[Computable continuous function]
Let $\R \subset \RN^{m}$ be a computable rectangle.
A function $f : \R \to\RN$ is said to be \emph{effectively uniformly continuous} if there exits a recursive function $\Op{d} : \NN \to \NN$ such that for every $k\in\NN$ and $n = \Op{d}(k)$
\begin{equation*}
	\left\|\bx_{1} - \bx_{2} \right\| \leq 2^{-n}
	\quad\text{implies}\quad
	\left|f(\bx_{1}) - f(\bx_{2})\right| \leq 2^{-k}\,.
\end{equation*} 
A function that is Banach--Mazur computable and effectively uniformly continuous is called a \emph{computable continuous function} and
we write $\Cc(\R)$ for the set of all computable continuous functions on the rectangle $\R$.
\end{definition}
Similarly, $\Cc^{k}(\R)$ stands for the set of all $k$-times continuously differentiable computable functions whose partial derivatives up to order $k$ are all computable continuous functions.

% =====================================================================================================
% ========== Computability of Iterative Algorithms ====================================================
% =====================================================================================================
\section{Optimization of smooth functions}
\label{sec:OptimizationProblem}

This section explains in greater detail the iterative optimization algorithm that will be studied in this paper and introduces some more notation needed to formulate and to prove our main results.

\subsection{Minimum value and minimizer}

We consider the following general minimization problem with a so-called box constraint \cite{Boyd_ConvexOpt}:
Let
$$\R = \R_{1} \times \R_{2} \times \dots \times \R_{r} \subset \RN^{m}$$
be a computable rectangle with $\R_{\ell} \subset \RN^{n_{\ell}}$ and $\sum^{r}_{\ell=1} n_{\ell} = m$, 
and let $f : \R \to \RN$ be a continuous function on $\R$.
Then we may ask for the \emph{minimum value} of $f$ on the rectangle $\R$, i.e. for the value
\begin{equation}
\label{equ:globalMin}
	\Op{Min}_{\R}(f)
	= \min_{\bx\in\R} f(\bx)
	= \!\!\!\min_{\substack{\bx_{\ell}\in\R_{l}, 1\leq\ell\leq r}} f(\bx_{1},\bx_{2},\cdots,\bx_{r}).
\end{equation}
Apart from the problem of finding the minimum value $\Op{Min}_{\R}(f)$ of $f$ on $\R$,
we may ask for a corresponding \emph{minimizer}, i.e. for a vector 
\begin{equation}
\label{equ:Minimizer}
		\widehat{\bx}\in\R
		\quad\text{such that}\quad
		f\left(\widehat{\bx}\right) = \min_{\bx\in\R} f(\bx)\,.
\end{equation}
The minimizer is generally not unique but there may be a whole set \eqref{equ:SetGlobalMinimizer} of \emph{global minimizers} in $\R$, i.e. a set of vectors that satisfy \eqref{equ:Minimizer}.

It depends on the actual application whether one needs to find the minimum value or the minimizer.
However, from the previous definition, it is clear that if one knows a minimizer $\widehat{\bx}$ then one also knows the minimum value $\Op{Min}_{\R}(f) = f(\widehat{\bx})$.
Conversely, knowing $\Op{Min}_{\R}(f)$ may not help in finding a corresponding minimizer $\widehat{\bx}$.
This observation indicates that the problem of finding the minimizer is usually harder than just determining the minimum value.
Unfortunately, in practical application the minimizer is often much more important than the minimum value.
In fact, the minimum value often has no particular meaning or significance, rather the point where this minimum value is attained is of importance.

\begin{remark}
Instead of the minimization problem, one may consider a corresponding maximization problem,
i.e. the problem finding the maximum value $\Op{Max}_{\R}(f)$ of $f$  or the maximizer of $f$ on $\R$.
Such a maximization problem can always be transformed into a minimization problem by considering the function $-f$ on $\mathcal{R}$. 
Then the minimizer of $-f$ is the maximizer of $f$.
So without loss of generality, this paper only discusses the minimization problem \eqref{equ:globalMin}.
\end{remark}

\subsection{Computability of the minimum value and the minimizer}

Apart from very special cases, there exists no closed-form solution for the minimum value $\Op{Min}_{\R}(f)$ or for the minimizer $\widehat{\bx}$ of a minimization problem.
Therefore, these values are usually approximated using numerical algorithms that determine a sequence that converge to the optimal value.
For a wide variety of optimization problems, algorithms are known that converge to the minimum value or minimizer, respectively.
Most notable are certainly the many different algorithms developed for convex optimization problems \cite{Boyd_ConvexOpt,Bertsekas_ConvexOpt,BenTalNemirovski_SIAM01}.
Nevertheless, from a practical point of view, the question is not only whether the algorithm converges to the optimum but whether this convergence is \emph{effective},
i.e. whether it is possible to control the approximation error and to stop algorithmically the computation if a desired error bound is achieved.
This problem of effective convergence is equivalent to the question of whether the minimum value or the minimizer are computable.

With respect to the computation of the  minimum value $\Op{Min}_{\R}(f)$, the following result concerning its computability is well known
(cf., e.g., \cite[Chapter~$6$]{PourEl_Computability}).

% ----- Minimum Value is Computable -----
\begin{proposition}
\label{prop:GlobalProblem}
Let $\R_{\ell} \subset \RN^{n_{\ell}}$, $\ell = 1,2,\dots,r$, be arbitrary computable rectangles, and let $\R = \R_{1}\times\dots\times\R_{r} \subset \RN^{m}$.
There exists a Turing machine $\TM_{\mathrm{Min}}$ that computes for every computable continuous function $f \in \Cc(\R)$ the value $\Op{Min}_{\R}(f)$.
\end{proposition}
% ---------------------------------------
\Cref{prop:GlobalProblem} shows that the minimum value \eqref{equ:globalMin} is always algorithmically computable on a digital computer provided $f$ is a computable continuous function.
Note in particular that the Turing machine $\TM_{\Op{Min}}$ in \Cref{prop:GlobalProblem} is \emph{universal} in the sense that it only depends on the rectangle $\R$. 
So for a fixed $\R$, the corresponding $\TM_{\mathrm{Min}}$ can compute $\Op{Min}_{\R}(f)$ for all $f \in \Cc(\R)$ as input.
Thus, if $\R$ and $\TM_{\mathrm{Min}}$ are fixed and if $f \in \Cc(\R)$ is arbitrary, 
then for every description of $f$ the Turing machine $\TM_{\mathrm{Min}}$ effectively computes a description of the real number $\Op{Min}_{\R}(f)$ \cite{LeeBoKu_CompOpt_24}.

With respect to the computation of the minimizer $\widehat{\bx}$, it is known  that if the minimization problem has a unique (global) minimizer, i.e. if the set \eqref{equ:SetGlobalMinimizer} contains only one vector, then this minimizer is always computable (cf., e.g., \cite[Chapter~$I.0.6$]{PourEl_Computability}).
If the minimizer is not unique then some (or all) minimizers might not be computable and there exist several examples of computable continuous functions $f$ that attain their minimum only at non-computable points (see, e.g., \cite{Specker_SatzVomMaximum} and references in \cite{PourEl_Computability}).

% ========== Iterative Algorithms =====================================================================
\subsection{Iterative optimization methods}
\label{sec:IterativeMethod}

The global optimization problem \eqref{equ:globalMin} that minimizes jointly over all $m$ components of $\bx\in\R \subset \RN^{m}$ is often considered as being too complex. 
Therefore, one applies block coordinate optimization methods of the Gauss--Seidel type that iteratively optimize over sub-rectangles $\R_{\ell}$, $\ell = 1,2,\dots, r$ while keeping the other variables fixed \cite{Beck_SIAMJOpt13, Powell_MathProgramm1973,Wright_MathProg15,Nesterov_SIAMOptm_12,Razaviyayn_SIAMJopt13,GrippoGaussSeidel2000,Xu_SIALJOpt,XuYin_SIAM_JIS13}.
Starting with an initial guess for the minimizer $\widetilde{\bx}^{(0)} = (\widetilde{\bx}_{1}^{(0)}, \widetilde{\bx}_{2}^{(0)}, \dots, \widetilde{\bx}_{r}^{(0)}) \in \R$, one solves for $k=0$ the optimization problems
\begin{equation}
\label{equ:ItarativLocalOpt}
	\widetilde{\bx}_{\ell}^{(k+1)} = \arg\min_{\by\in\R_{\ell}} f\big(\widetilde{\bx}_{1}^{(k+1)}, \dots, \widetilde{\bx}_{\ell-1}^{(k+1)},\by,\widetilde{\bx}_{\ell+1}^{(k)}, \dots, \widetilde{\bx}_{r}^{(k)}\big)
\end{equation}
successively for $\ell=1, 2, \dots, r$ and iterates over $k=1,2,\dots$.
In particular, if $n_{1} = n_{2} = ... = n_{\ell} = 1$ then each step optimizes only over one coordinate of the vector $\widetilde{\bx}^{(k)}\in\RN^{m}$ while leaving all other coordinates fixed,  (cf. \eqref{equ:ItarativLocalOpt_1D}).
This procedure yields a sequence $\left\{ \widetilde{\bx}^{(k)} \right\}_{k\in\NN}$ of approximations of a minimizer of the optimization problem \eqref{equ:globalMin}.
The components of each vector $\widetilde{\bx}^{(k)}$ are minimizers of a local optimization problem according to \eqref{equ:ItarativLocalOpt}. %i.e. they are obtained as a result of the $\arg\min$-operation in \eqref{equ:ItarativLocalOpt}.
Therefore we say that $\left\{ \widetilde{\bx}^{(k)} \right\}_{k\in\NN}$ is a \emph{sequence of local minimizers}.

One can show that under some mild conditions on $f$, this sequence of local minimizers converges to a global minimizer $\widehat{\bx}$ of the optimization problem \eqref{equ:globalMin}.
In fact, there exist many studies that investigate the convergence behavior of iterative algorithms for solving optimization problems of the form \eqref{equ:MinProbGen} (see, e.g., \cite{Beck_SIAMJOpt13,GrippoGaussSeidel2000,Razaviyayn_SIAMJopt13,XuYin_SIAM_JIS13}). 
All these investigations consider the convergence of sequences $\left\{ \widetilde{\bx}^{(k)} \right\}_{k\in\NN}$ that were obtained by an iterative optimization algorithm. 
The strongest convergence statements of these investigations have typically the following form:
Let $\left\{ \widetilde{\bx}^{(k)} \right\}_{k\in\NN}$ be a sequence of local minimizers determined by an iterative optimization algorithm.
Then, under some conditions on $f$, this sequence has at least one limit point which implies (see, e.g., \cite{GrippoGaussSeidel2000}) that all limit points of this sequence are critical points of \eqref{equ:MinProbGen}.
Apart from results regarding the convergence of iterative optimization algorithms there seems to exist no estimates on the convergence speed of these algorithms.
However, such results are highly desirable from a practical point of view because the convergence results alone imply, in principle, that the iterative algorithm has to compute \emph{ad infinitum} to reach the optimal value.
In practice, however, one needs a criterion to stop the iteration if a desired error bound is achieved,
i.e. one needs the possibility to pass an integer $M\in\NN$ to the algorithm such that the algorithm is able to stop the iteration at $K\in\NN$ if $\left|\widehat{\bx} - \widetilde{\bx}^{(K)}\right| < 2^{-M}$.
Up to now, no such algorithm is known and the following results will show that generally no such algorithm can exist, even in the simple case $m=2$.

\begin{remark}
The non-existence of such an algorithmic stopping criterion was recently observed in several central and concrete problems in information theory.
One example is the celebrated \emph{Blahut--Arimoto algorithm} \cite{Blahut_TIT72,Arimoto_TIT72}. It computes an infinite sequence of input distributions of a channel that converges to the capacity achieving input distribution \cite{Csizar_TIT74}.
Since its invention, researchers tried to find a computable stopping algorithm that is able to stop the iteration based on a required approximation error. 
To date, no such algorithm was found and \cite{BSP_TIT23} showed that no such computable stopping exists for the Blahut--Arimoto algorithm.
We refer to \cite{LeeBK_TIT24} for more information theoretic questions that show a similar behavior. 
\end{remark}

Presupposing a sequence of local minimizers $\left\{ \widetilde{\bx}^{(k)} \right\}_{k\in\NN} \subset \RN^{m}$ converges to a global minimizer $\widehat{\bx}$ of \eqref{equ:globalMin}, this paper asks whether this convergence is always effective, i.e. whether we are able to algorithmically stop the iteration if a predefined approximation error is achieved.

\begin{question}
\label{question:effectiveConv}
Given a sequence $\big\{ \widetilde{\bx}^{(k)} \big\}_{k\in\NN}$ of local minimizers that converges to a global minimizer $\widehat{\bx}$ of \eqref{equ:globalMin}.
Is this convergence always effective?
\end{question}

As a second problem, we note that the $\arg\min$-operation in Step~\eqref{equ:ItarativLocalOpt} of the iterative algorithm is more like a pseudo-code.
In this form, it is not clear whether there exists an effective implementation for this operation on digital hardware.
The answer will, of course, strongly depend on $f$ and $\R_{\ell}$ and raises the following general question.

\begin{question}
\label{question:Implement_argMin}
Does there always exist an effective implementation of Step~\eqref{equ:ItarativLocalOpt}?
\end{question}

Associated with the block coordinate optimization method \eqref{equ:ItarativLocalOpt}, we define for every $\ell = 1,2,\dots,r$, the sets
\begin{align}
\label{equ:set_Minl}
	\M_{\ell}
	&= \M_{\ell}(\bx_{1},\dots,\bx_{\ell-1},\bx_{\ell+1},\dots,\bx_{r})\nonumber\\
	&= \Big\{ \widehat{\bx}_{\ell}\in\R_{\ell}\ :
	f(\bx_{1},\dots,\bx_{\ell-1},\widehat{\bx}_{\ell},\bx_{\ell+1},\dots,\bx_{r})\nonumber\\
	&= \min_{\by\in\R_{\ell}} f(\bx_{1},\dots,\bx_{\ell-1},\by,\bx_{\ell+1},\dots,\bx_{r}) \Big\}
\end{align}
of all \emph{local minimizes} with respect to the $\ell$th variable for fixed variables $\bx_{1}$, $\dots$, $\bx_{\ell-1}$, $\bx_{\ell+1}$, $\dots$, $\bx_{r}$.
So the set $\M_{\ell} = \M_{\ell}(\bx_{1},\dots,\bx_{\ell-1},\bx_{\ell+1},\dots,\bx_{r})$ contains all $\widehat{\bx}\in\R_{\ell}$ that minimize the right hand side of \eqref{equ:ItarativLocalOpt}.
Then the $\arg\min$-operator in \eqref{equ:ItarativLocalOpt} simply chooses one element from the set $\M_{\ell}$.
This operation can be described by a so-called assignment function:

\begin{definition}[Assignment function]
\label{def:AssignFkt}
Consider the optimization problem \eqref{equ:globalMin} for a function $f : \R \to\RN$.
A function $G_{\ell} : \R_{1}\times\dots\times\R_{\ell-1}\times\R_{\ell+1}\times\dots\times\R_{r} \to \R_{\ell}$
is said to be an \emph{assignment function of $f$} for the $\ell$th coordinate of the iterative optimization procedure if it has the property
\begin{equation}
\label{equ:DefOfGl}
	G_{\ell}(\bx_{1},\dots,\bx_{\ell-1},\bx_{\ell+1},\dots,\bx_{r}) \in \M_{\ell}\,.
\end{equation}
The set of all assignment functions of $f$ for the $\ell$th coordinate is denoted by $\mathcal{A}_{\ell}(f)$.
\end{definition}
Using this notion, the iteration step \eqref{equ:ItarativLocalOpt} of the optimization problem can be written as
\begin{equation}
\label{equ:OptStepWithG}
	\bx_{\ell}^{(k+1)}
	= G_{\ell}\big(\bx_{1}^{(k+1)}, \dots, \bx_{\ell-1}^{(k+1)},\bx_{\ell+1}^{(k)}, \dots, \bx_{r}^{(k)}\big)\,.
\end{equation}
In particular, the coordinate-wise optimization where the dimension of each rectangle $\R_{\ell}$ is equal to one, 
can be rewritten as shown in Algorithm~\ref{alg:CoordinateWise2}.
% ----- Algorithm -----
\begin{algorithm}
\caption{Coordinate-wise optimization}
\label{alg:CoordinateWise2}
\begin{algorithmic}[1]
\STATE{Initialize $\bx^{(0)} = (x^{(0)}_{1},\cdots,x^{(0)}_{m}) \in \RN^{m}$ and $k=0$}
\REPEAT 
\FOR{$\ell = 1,2,\dots,m$} %{$m$}
\STATE $x^{(k+1)}_{\ell} = G_{\ell}(x^{(k+1)}_{1}, \dots, x^{(k+1)}_{\ell-1},x^{(k)}_{\ell+1}, \dots, x^{(k)}_{m})$
\ENDFOR
\STATE $k = k+1$
\UNTIL{Convergence}
%%\UNTIL{Convergence}
%%\Output{$\bx_{\mathrm{out}} = \bx^{(k)}$}
\RETURN{$\bx^{(k)}$}
\end{algorithmic}
\end{algorithm}
% ---------------------

\begin{remark}
Note that for every $\ell \in \left\{ 1,2,\dots, r \right\}$ the set $\A_{\ell}(f)$ contains generally many different assignment functions.
Namely, there are as many different functions $G_{\ell}$ as there are different vectors in $\M_{\ell}$.
In principle, one can choose any $G_{\ell} \in \mathcal{A}_{\ell}(f)$ for the optimization step \eqref{equ:OptStepWithG}.
However, in order that step \eqref{equ:ItarativLocalOpt} be algorithmically solvable on a digital computer, we need to choose $G_{\ell}\in\mathcal{A}_{\ell}(f)$ to be a computable function.
The interesting question is now whether this is always possible.
The following section will show that there exist very simple examples of computable continuous functions $f$ such that for some $\ell\in \{1,2,\dots,r\}$, the set $\mathcal{A}_{\ell}(f)$ contains no computable assignment function.
\end{remark}

% =====================================================================================================
% ========== Computability of Iterative Algorithms ====================================================
% =====================================================================================================
\section{Algorithmic computability of assignment functions}
\label{sec:ProblemWithAssignmentFunc}

To make our arguments as clear as possible,
we consider the simplest case of the general optimization problem \eqref{equ:globalMin}. 
Namely we consider functions $f : \RN^{2} \to \RN$ on the rectangle $\R = \R_{1}\times\R_{2}$ with $\R_{1} = [-a,a]$ and $\R_{2} = [-b,b]$ for arbitrary positive computable numbers $a,b\in\RNc$, i.e. we consider the minimization problem
\begin{equation}
\label{equ:Min2D}
	\min_{(x_{1},x_{2}) \in [-a,a]\times [-b,b]} f(x_{1}, x_{2})\,.
\end{equation}
In view of the general optimization problem stated at the beginning of \Cref{sec:OptimizationProblem}, we thus have $r = 2$ and $n_{1} = n_{2} = 1$, and the corresponding iterative optimization algorithm is a coordinate-wise optimization as shown in Algorithm~\ref{alg:CoordinateWise2}.

\subsection{A function with no computable assignment function}

Our first theorem gives an example of a computable continuous function $f_{1}$ of two variables that is convex in each variable but such that the corresponding set $\A_{1}(f_{1})$ of assignment functions for the first step in the coordinate wise optimization algorithm contains no computable assignment function $G_{1}$.

\begin{theorem}
\label{thm:MainThm2D}
Let $a,b\in\RNc$ with $a>1$ and $b>0$ be arbitrary and let $\R = \R_{1}\times \R_{2}$ with $\R_{1} = [-a,a]$ and $\R_{2} = [-b,b]$.
There exists a computable continuous function $f_{1} : \R \to \RN$ with the following properties:
\begin{enumerate}
\item 
For every fixed $x_{2}\in\R_{2}$, 
the function $f_{1}(\cdot,x_{2}) : \R_{1}\to\RN$ is a computable continuous function that is convex and piecewise linear (with 3 linear pieces).
\item
For every fixed $x_{1}\in\R_{1}$,
the function $f_{1}(x_{1},\cdot) : \R_{2}\to\RN$ is a computable continuous function that is convex and piecewise linear (with 2 linear pieces).
\item
The function $f_{1}$ has only global minima (i.e. it has no local minima that are not global) and the set of all global minima is convex.
\item
For every $x_{2}\in\R_{2}$, $x_{2}\neq 0$, the function $f(\cdot,x_{2})$ has only one global minimum.
\item
All assignment functions $G_{1} \in\A_{1}(f_{1})$ are not Turing computable.
\end{enumerate}
\end{theorem}

\begin{remark}
The assumption $a>1$ is not a restriction of the generality.
For $a < 1$ the construction in the subsequent proof has to be adapted in an obvious way.
\end{remark}

\begin{remark}
The function $f_{1}$, constructed in \Cref{thm:MainThm2D}, has very good properties. In particular, for every fixed $x_{1} \in \RNc$, $f_{1}(x_{1},\cdot)$ is a computable continuous, convex function, and for every fixed $x_{2}\in \RNc$, $f_{1}(\cdot,x_{2})$ is a computable continuous, convex function.
Moreover, \Cref{thm:MainThm2D} shows that $f_{1}$ has only global minimizers which all lie inside $\R$. It has no further local minimizers in $\R$ or even in $\RN^{2}$.
This implies that for $f_{1}$ the points of convergence of the algorithms from \cite{GrippoGaussSeidel2000,XuYin_SIAM_JIS13} are not only critical points according to \eqref{equ:defCriticalPoint} but they are the global minimizers of $f_{1}$.
Consequently \Cref{thm:MainThm2D} implies that the proposed algorithm from \cite{GrippoGaussSeidel2000,XuYin_SIAM_JIS13} cannot be implemented algorithmically on digital hardware (i.e. on a Turing machine) for the function $f_{1}$.
This is because the $\arg\min$-operators from \cite{GrippoGaussSeidel2000,XuYin_SIAM_JIS13}, which correspond to our assignment functions of $f_{1}$ (cf. \Cref{def:AssignFkt}), are not Turing computable.
Nevertheless, the subsequent proof will show that the assignment functions of $f_{1}$ (i.e. the $\arg\min$-operators) are well defined and fairly simple functions, namely step functions.
They are just not Turing computable.
\end{remark}

\begin{remark}
\label{rem:Glattheit_f1}
By a simple change in the construction of $f_{1}$ in the subsequent proof of \Cref{thm:MainThm2D},
it is possible to replace the piecewise linear function $f_{1}$ in \Cref{thm:MainThm2D} by a function $f_{1} \in \C^{K}(\RN)$, where $K\in\NN$ is arbitrary and where all partial derivatives $\frac{\partial^{\ell + k} f_{1}}{\partial x^{\ell}_{1} \partial x^{k}_{2}}$ with $\ell + k \leq K$ are computable continuous functions.
\end{remark}

As a consequence of \Cref{thm:MainThm2D},
one immediately obtains that the first optimization step, which should find the local minimum with respect to the first coordinate of $\bx$, cannot be solved algorithmically.

\begin{corollary}
\label{cor:Func1}
Let $f_{1} : \R_{1}\times\R_{2} \to \RN$ be the computable continuous function of \Cref{thm:MainThm2D}.
Then the optimization step
\begin{equation*}
	x^{(k+1)}_{1}
	= \arg\min_{y\in\R_{1}} f_{1}(y,x^{(k)}_{2})\,,
	\qquad k\in\NN\,.
\end{equation*}
cannot be solved algorithmically, i.e. there exists no Turing machine that is able to compute $x_{1}^{(k+1)}$ on input $x^{(k)}_{2}$.
\end{corollary}

\begin{remark}
Note that \Cref{cor:Func1} answers \Cref{question:Implement_argMin} negatively.
\end{remark}

\begin{proof}[Proof of \Cref{thm:MainThm2D}]
We explicitly define the function $f_{1}$ on the whole plane $\RN\times \RN$ 
as follows:\\
For all $x_{2} \leq 0$ by
\begin{equation*}
	f(x_{1},x_{2})
	= \left\{\begin{array}{lll}
	\left(\tfrac{x_{2}}{2}-1\right) x_{1} - \left(\tfrac{3}{2}x_{2} + 1\right) & : & x_{1} < -1\\[1ex]
	\tfrac{1}{2} \left( x_{2} x_{1} - 3 x_{2}\right) & : & \left|x_{1}\right| \leq 1\\[1ex]
	\left(1 - \tfrac{x_{2}}{2}\right) x_{1} - \left(\tfrac{1}{2}x_{2} + 1\right) & : & x_{1} > 1
	\end{array}\right. ,
\end{equation*}
and for all $x_{2} \geq 0$ by
\begin{equation*}
	f(x_{1},x_{2})
	= \left\{\begin{array}{lll}
	-\left(1 + \tfrac{x_{2}}{2}\right) x_{1} + \left(\tfrac{1}{2}x_{2} - 1\right) & : & x_{1} < -1\\[1ex]
	\tfrac{1}{2} \left( x_{2}x_{1} + 3 x_{2} \right) & : & \left|x_{1}\right| \leq 1\\[1ex]
	\left(1 + \tfrac{x_{2}}{2}\right) x_{1} + \left(\tfrac{3}{2}x_{2} - 1\right) & : & x_{1} > 1
	\end{array}\right. .
\end{equation*}
For illustration purposes, the so defined function $f_{1}$ is shown in \Cref{fig:ProofThm1Plot1}. %and \ref{fig:ProofThm1Plot2}.

Properties~1) and 2) are immediately clear from the previous definition of $f_{1}$.
This definition shows in particular:
\begin{enumerate}
\item[i)] $f_{1}$ is a computable continuous function on $\R$.
\item[ii)] $f_{1}(x_{1},x_{2}) \geq 0$ for all $(x_{1},x_{2}) \in \R$.
\item[iii)] For every $x_{2}\in\R_{2}$, $x_{2}\neq 0$ we have $f_{1}(x_{1},x_{2}) > 0$ for all $x_{1}\in\R_{1}$. 
\item[iv)] For every $x_{2}\in\R_{2}$, $x_{2}\neq 0$ holds
\begin{equation}
\label{equ:Proof_Inequ1}
	f_{1}(x_{1},x_{2}) > f_{1}(x_{1},0)\,,
	\quad \text{for all}\ x_{1}\in\R_{1}\,.
\end{equation}
\end{enumerate}
Points ii)-iv) imply that all global minima lie on the line $x_{2} = 0$, 
that $f_{1}$ has no local minimum that is not a global minimum, and 
\begin{equation*}
	\M_{\R}(f_{1}) = \big\{ (x_{1},0) \in \R\ :\ x_{1}\in [-1,1] \big\}\,.
\end{equation*}
This set is certainly convex, which proves Property 3).
Next, we determine the sets \eqref{equ:set_Minl} of local minimizers for the function $f_{1}$.
To this end, we notice that for a fixed $x_{2} > 0$, the function $f_{1}(\cdot,x_{2})$ is monotonically decreasing on $[-a,-1]$ and monotonically increasing on $[-1,a)$.
For $x_{2} < 0$, the function $f_{1}(\cdot,x_{2})$ is monotonically decreasing on $[-a,1]$ and monotonically increasing on $[1,a)$.
For $x_{2} = 0$, $f_{1}(\cdot,x_{2})$ is monotonically decreasing on $[-a,-1]$, equal to zero on $[-1,1]$, and monotonically increasing on $[1,a)$.
All this implies that 
\begin{equation*}
	\M_{1}(x_{2}) = \left\{\begin{array}{lll}
	\left\{ 1 \right\} & : & x_{2} < 0\\
	\left\{ x\in\RN : |x|\leq 1\right\} & : & x_{2} = 0\\
	\left\{ -1 \right\} & : & x_{2} > 0
	\end{array}\right.\,.
\end{equation*}
Moreover, \eqref{equ:Proof_Inequ1} implies immediately $\M_{2}(x_{1}) = \left\{ 0\right\}$ for all $x_{1}\in\R_{1}$.
Therewith, we can determine the corresponding assignment functions $G_{\ell}$, defined by \eqref{equ:DefOfGl}.
For $G_1$, we have
\begin{equation}
\label{equ:G1proof}
	G_{1}(x_{2})
	= G_{1,\alpha}(x_{2}) = \left\{\begin{array}{lll}
	1 & : & x_{2} \in (-\infty,0)\\
	\alpha & : & x_{2} = 0\\
	-1 & : & x_{2} \in (0,\infty)
	\end{array}\right.
\end{equation}
for some arbitrary $\alpha\in [-1,1]$.
This means that for every $\alpha \in [-1,-1]$ there is a different function $G_{1,\alpha}$ and all these functions differ only by their value at $x_{2} = 0$.
So the set of assignment functions of $f_{1}$ for the first coordinate
\begin{equation}
\label{equ:Set_A1}
	\mathcal{A}_{1}(f_{1}) = \big\{ G_{1,\alpha}\ :\ \alpha\in\RN, \alpha\in [-1,1] \big\}
\end{equation}
contains uncountably many functions.
For the second coordinate, there exists only one assignment function, given by
$G_{2}(x_{1}) = 0$ for all $x_{1}\in\R_{1}$.

All functions in $\A_{1}(f_{1})$ are step functions and so \Cref{lem:Galpha} implies that every $G_{1,\alpha}$, $\alpha\in [-1,1]$ is not a Borel--Turing computable functions.
This finally proves Property 5).
\end{proof}

% ----- Figure 1 -----
\begin{figure}[t]
\centering
\subfloat{\includegraphics[scale=0.40]{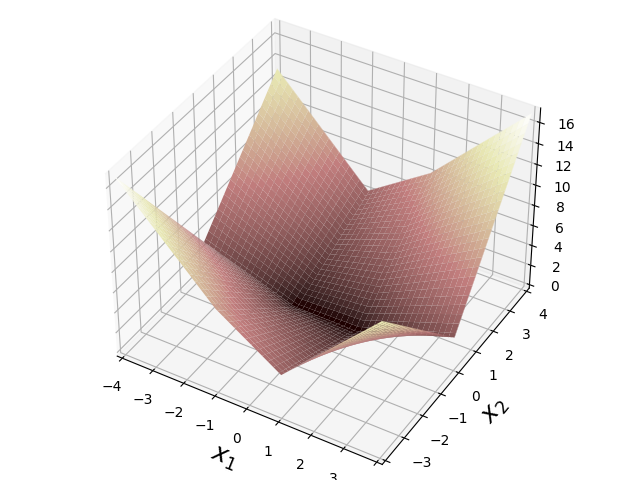}}
\subfloat{\includegraphics[scale=0.40]{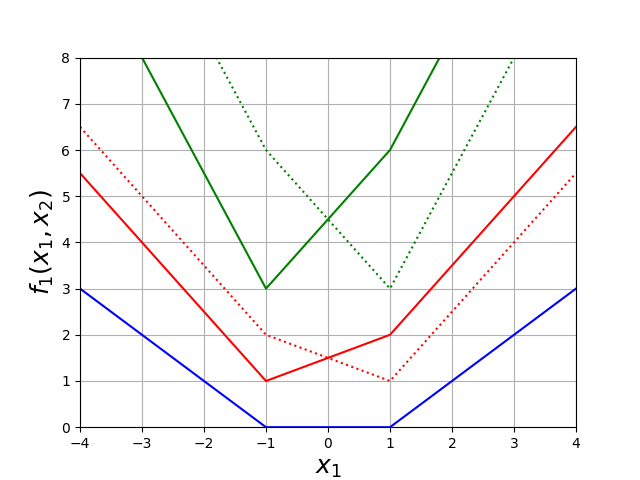}}
\caption{Illustration of the function $f_{1}$ constructed in the proof of \Cref{thm:MainThm2D}. The plot on the right shows the function $f_{1}(x_{1},x_{2})$ (solid line) and $f_{1}(x_{1},-x_{2})$ (dotted line) for fixed $x_{2} = 0$ (blue), $x_{2} = \pm 1.0$ (red), and $x_{2} = \pm 3.0$ (green). }
\label{fig:ProofThm1Plot1}
\end{figure}
% --------------------

\begin{proof}[Proof of \Cref{cor:Func1}]
Contrary to the statement of \Cref{cor:Func1}, assume that there exists a Turing machine $\TM$ that solves the following problem for any arbitrary $x_{2} \in\RNc\cap\R_{2}$ and for a computable $x_{1} \in \M_{1}(x_{2})$:
For every representation of $x_{2}$, $\TM$ determines a representation of $x_{1}$.
This would imply that the function $G_{1} = G_{1,\alpha}$ with $\alpha = x_{1} \in [-1,1]$ is Borel--Turing computable which contradicts the statement of \Cref{thm:MainThm2D}.
\end{proof}

\begin{remark}
By our construction of $f_{1}$ all global minimizers of \eqref{equ:Min2D} lie inside the rectangle $\R = [-a,a] \times [-b,b]$.
This implies (cf., e.g., \cite{GrippoGaussSeidel2000}) that all critical points of Problem~\eqref{equ:Min2D},
are exactly the global minimizers of \eqref{equ:Min2D}.
Then the results of \cite{GrippoGaussSeidel2000} imply that every arbitrary sequence $\left\{ \widetilde{\bx}_{n} \right\}_{n\in\NN}$ of local minimizers converges to a critical point of \eqref{equ:Min2D}.
However, according to \Cref{cor:Func1}, already the first iterative steps of Algorithm~\ref{alg:CoordinateWise2} cannot effectively be solved on a digital computer,
because the corresponding assignment functions are all not computable.
\end{remark}

% =======================================================================
% ========== Disskussion Convergence of the Iterative Procedure ==========
% =======================================================================
\subsection{Convergence of the iterative procedure for $f_{1}$}
\label{sec:Convergence}

This subsection will show that if the assignment functions $G_{1}$ and $G_{2}$ would be computable in each step, then the iterative procedure (Algorithm~\ref{alg:CoordinateWise2}) would indeed converges to a global minimizer of $f_{1}$. In fact, we will see that the iterative procedure converges for any initialization vector $\bx^{(0)} = (x^{(0)}_{1}, x^{(0)}_{2}) \in \RN^{2}$ in at most two steps.

Let $f_{1} : \RN\times\RN \to \RN$ be given as in \Cref{thm:MainThm2D}, 
$G_{1} \in \A_{1}(f_{1})$ be an arbitrary assignment function from the set \eqref{equ:Set_A1},
and $G_{2}$ be the unique assignment function from set $\A_{2}(f_{1})$, given by $G_{2}(x_{1}) = 0$ for all $x_{1}\in\R_{1}$ (cf. the proof of \Cref{thm:MainThm2D}).
If $\bx^{(0)} = (x^{(0)}_{1},x^{(0)}_{2})\in\R_{1}\times\R_{2}$ is an arbitrary initialization vector.
Then the iterative procedure \eqref{equ:ItarativLocalOpt} computes for $k=0,1,2,\dots$ 
\begin{equation*}
	x^{(k+1)}_{1} = G_{1}\big(x^{(k)}_{2}\big)
	\qquad\text{and}\qquad
	x^{(k+1)}_{2} = G_{2}\big(x^{(k)}\big)\,.
\end{equation*}
We distinguish three different cases for the initial vector:

I) For $x^{(0)}_{2} > 0$, the iterative procedure computes successively:
\begin{align*}
	&x^{(1)}_{1} = G_{1}(x^{(0)}_{2}) = -1\,, &x^{(1)}_{2} = G_{2}(x^{(1)}_{1}) = 0\\
	&x^{(2)}_{1} = G_{1}(x^{(1)}_{2}) = G_{1}(0)\,, &x^{(2)}_{2} = G_{2}(x^{(2)}_{1}) = 0\\
	&x^{(3)}_{1} = G_{1}(x^{(2)}_{2}) = G_{1}(0)\,, &x^{(3)}_{2} = G_{2}(x^{(3)}_{1}) = 0\\
	& \qquad\vdots
\end{align*}
So the procedure converges after at most $2$ steps, no matter which particular $G_{1} = G_{1,\alpha}$, $\alpha\in [-1,-1]$ was chosen. 
For $\alpha = -1$, we even have $G_{1}(0) = -1$, and so the procedure already converges after the first step.

II) For $x^{(0)}_{2} < 0$, the iterative procedure gives
\begin{align*}
	&x^{(1)}_{1} = G_{1}(x^{(0)}_{2}) = 1\,, &x^{(1)}_{2} = G_{2}(x^{(1)}_{1}) = 0\\
	&x^{(2)}_{1} = G_{1}(x^{(1)}_{2}) = G_{1}(0)\,, &x^{(2)}_{2} = G_{2}(x^{(2)}_{1}) = 0\\
	&x^{(3)}_{1} = G_{1}(x^{(2)}_{2}) = G_{1}(0)\,, &x^{(3)}_{2} = G_{2}(x^{(3)}_{1}) = 0\\
	& \qquad\vdots
\end{align*}
Thus, the algorithm converges after at most $2$ step, and for $\alpha = 1$, i.e. for $G_{1}(0) = 1$, the algorithm converges already after the first step.

III) For $x^{(0)}_{2} = 0$, we simply get
\begin{align*}
	&x^{(1)}_{1} = G_{1}(x^{(0)}_{2}) = G_{1}(0)\,, &x^{(1)}_{2} = G_{2}(x^{(1)}_{1}) = 0\\
	&x^{(2)}_{1} = G_{1}(x^{(1)}_{2}) = G_{1}(0)\,, &x^{(2)}_{2} = G_{2}(x^{(2)}_{1}) = 0\\
	& \qquad\vdots
\end{align*}
Thus, the algorithm already converges after the first step.

% =======================================================================
% ========== Diskussion Approximation of the decision function ==========
% =======================================================================
\subsection{Approximation of the assignment function}
\label{sec:ApproximationAssFkt}

The main problem in solving the minimization problem \eqref{equ:Min2D} for the function $f_{1}$ given in \Cref{thm:MainThm2D} using the iterative coordinate-wise Algorithm~\ref{alg:CoordinateWise2} arises from the non-computability of the assignment functions $G_{1}\in \A_{1}(f_{1})$.
This raises the question of whether it is possible to replace the non-computable function by an appropriate computable function $\widetilde{G}_{1}$ that allows one to compute an approximation $\widetilde{x}^{(k+1)}_{1} = \widetilde{G}_{1}(x^{(k)}_{2})$ of $x^{(k+1)}_{1} = G_{1}(x^{(k)}_{2})$.

We will show that this is generally impossible.
In fact, even if we merely require that the error $\big| \widetilde{x}^{(k+1)}_{1} - x^{(k+1)}_{1} \big|$ is upper bounded by a fixed constant $\epsilon \leq \max(1,a)$ there exists no computable function $\widetilde{G}_{1}$ of $G_{1}$ that guarantees 
\begin{equation*}
	\left| \widetilde{x}^{(k+1)}_{1} - x^{(k+1)}_{1} \right|
	= \left|  \widetilde{G}_{1}(x^{(k)}_{2}) - G_{1}(x^{(k)}_{2}) \right|
	< \epsilon
\end{equation*}   
at every iteration step $k$.                
Of course, if it is not possible to guarantee a fixed error bound then it is \emph{a fortiori} not possible to satisfy Turing's requirement of an effective approximation that satisfies any arbitrary small approximation error.

The following theorem proves that the described approximation of the assignment function is not possible.

\begin{theorem}
Let $a,b\in\RNc$, $a,b > 0$ and consider the minimization problem on the rectangle $\R = [-a,a]\times [-b,b]$ for the function $f_{1}$, constructed in \Cref{thm:MainThm2D}.
Let $G_{1}\in\A_{1}(f_{1})$ be an arbitrary assignment function \eqref{equ:G1proof} and let $G : [-b,b] \to\RN$ be a function that satisfies
\begin{equation*}	
	\sup_{x_{2}\in [-b,b]} \left| G_{1}(x_{2}) - G(x_{2}) \right|
	< \min(1,a)\,.
\end{equation*}
Then $G$ is not Turing computable.
\end{theorem}

This theorem shows that it is not possible to approximate a non-computable assignment function by a computable function such that the maximum error is less than $a$ (if $a \leq 1$) or less than $1$ (if $a > 1$).

\begin{proof}
Let $G_{1} : [-b,b] \to [-a,a]$ and $G_{2} : [-a,a]\to [-b,b]$ be assignment functions of $f_{1}$ as derived in the proof of \Cref{thm:MainThm2D}.
More precisely, $G_{1}$ is given by \eqref{equ:G1proof} and $G_{2}$ is simply the zero function: $G_{2}(x_{1}) = 0$ for all $x_{1}\in [-a,a]$.

First we consider the case $a\geq 1$.
Assume that there exists an assignment function $G_{1}\in\A_{1}(f_{1})$ such that there is a Turing computable function $G_{*}$ that satisfies
\begin{equation*}
	\sup_{x_{2}\in [-b,b]} \big| G_{1}(x_{2}) - G_{*}(x_{2}) \big| = \gamma < 1\,.
\end{equation*}
Then for all $x_{2} \in [-b,0)$, we have $\left| 1 - G_{*}(x_{2}) \right| \leq \gamma$ which implies that $G_{*}(x_{2}) \geq 1-\gamma =: \delta >0$.
Similarly, for all $x_{2} \in (0,b]$, we have $\left| -1 - G_{*}(x_{2}) \right| = \left| 1 + G_{*}(x_{2}) \right| \leq \gamma$ which implies $G_{*}(x_{2}) \leq \gamma - 1 = -\delta < 0$.
So $G_{*}(x_{2})$ is discontinuous at $x_{2} = 0$.
Therefore, one can show (in the same way as in the proof of \Cref{lem:Galpha}) that $G_{*}$ is not Banach--Mazur computable and therefore also not Turing computable.

Second we consider the case $a < 0$. 
In this case, only the assignment function $G_{1}$ is slightly different from the case $a \geq 1$.
Indeed, for all $x_{2} < 0$, one has (cf. \Cref{fig:ProofThm1Plot1} for illustration)
$f_{1}(x_{1},x_{2}) \geq f_{1}(a,x_{2})$ for all $x_{1}\in [-a,a]$ and 
\begin{equation*}
	\min_{x_{1}\in [-a,a]} f_{1}(x_{1},x_{2}) = f_{1}(a,x_{2})\,.
\end{equation*}
Similarly, for all $x_{2} > 0$, one has $f_{1}(x_{1},x_{2}) \geq f_{1}(-a,x_{2})$ for all $x_{1}\in [-a,a]$ and
\begin{equation*}
	\min_{x_{1}\in [-a,a]} f_{1}(x_{1},x_{2}) = f_{1}(-a,x_{2})\,.
\end{equation*}
Moreover, every $x_{1} \in [-a,a]$ is a minimizer of the function $f_{1}(\cdot,0)$.
Consequently, the assignment function $G_{1}$ becomes
\begin{equation*}
	G_{1}(x_{2})
	= G_{1,\alpha}(x_{2}) = \left\{\begin{array}{rll}
	a & : & x_{2} \in [-b,0)\\
	\alpha & : & x_{2} = 0\\
	-a & : & x_{2} \in (0,b]
	\end{array}\right.
\end{equation*}
for some arbitrary $\alpha\in [-a,a]$.
Now, we can apply the same arguments as for $a < 0$ to show that $G_{1}$ cannot be approximated by a computable function.
\end{proof}

% =====================================================================================================
% ========== Sufficient Conditions ====================================================================
% =====================================================================================================
\section{Reachability of global minimizers}
\label{sec:Reachability}

We still consider the optimization problem \eqref{equ:globalMin}
on a computable rectangle $\R = \R_{1}\times\R_{2}$ with $\R_{1} = [-a,a]$ and $\R_{2} = [-b,b]$ for some positive $a,b\in\RNc$ and with a given function $f : \RN^{2} \to \RN$.
This section studies the behavior of the assignment functions $G_{1}$ and $G_{2}$ associated with the iterative optimization problem in more detail.
In particular, we consider the point sets
\begin{equation*}
	\left\{ ( G_{1}(x_{2}) , x_{2} ) \subset \RN^{2} : x_{2} \in \R_{2} \right\}
	\qquad\text{and}\qquad
	\left\{ ( x_{1} , G_{2}(x_{1}) ) \subset \RN^{2} : x_{1} \in \R_{1} \right\}\,.
\end{equation*}
By the \Cref{def:AssignFkt} of the assignment functions, we certainly have
\begin{align*}
		G_{1}(x_{2}) \in \M_{1}(x_{2}) &= \arg\min_{x_{1}\in\R_{1}} f(x_{1},x_{2})\quad\text{and}\\
		G_{2}(x_{1}) \in \M_{2}(x_{1}) &= \arg\min_{x_{2}\in\R_{2}} f(x_{1},x_{2})\,.
\end{align*}
Let $\widehat{\bx} = (\widehat{x}_{1} , \widehat{x}_{2}) \in \M_{\R}(f)$ be an arbitrary global minimizer of $f$.
We want to study the behavior of the points $(G_{1}(x_{2}) , x_{2})\in\RN^{2}$ as $x_{2}$ approaches $\widehat{x}_{2}$
and the behavior of the points $( x_{1} , G_{2}(x_{1}) )\in\RN^{2}$ as $x_{1}$ approaches $\widehat{x}_{1}$.
To this end, we define the sets 
\begin{equation*}
	\mathcal{G}^{+}_{1}(\widehat{x}_{2})
	= \left\{ \bx = \lim_{\substack{x_{2} \to \widehat{x}_{2}\\x_{2} > \widehat{x}_{2}}} \big(G_{1}(x_{2}) , x_{2} \big)\ :\ G_{1} \in \A_{1}(f) \right\}
\end{equation*}
and
\begin{equation*}
	\mathcal{G}^{-}_{1}(\widehat{x}_{2})
	= \left\{ \bx = \lim_{\substack{x_{2} \to \widehat{x}_{2}\\ x_{2} < \widehat{x}_{2}}} \big(G_{1}(x_{2}) , x_{2} \big)\ :\ G_{1} \in \A_{1}(f) \right\}\,,
\end{equation*}
i.e. the set of all limits of points $\left(G_{1}(x_{2}) , x_{2} \right) \in \RN^{2}$ as $x_{2}$ converges to $\widehat{x}_{2}$ from above and below, respectively.
Similarly, we define the sets
\begin{equation*}
	\mathcal{G}^{+}_{2}(\widehat{x}_{1})
	= \left\{ \bx = \lim_{\substack{x_{1} \to \widehat{x}_{1}\\x_{1} > \widehat{x}_{1}}} \big( x_{1}, G_{2}(x_{1}) \big)\ :\ G_{2} \in \A_{2}(f) \right\}
\end{equation*}
and
\begin{equation*}
	\mathcal{G}^{-}_{2}(\widehat{x}_{1})
	= \left\{ \bx = \lim_{\substack{x_{1} \to \widehat{x}_{1}\\ x_{1} < \widehat{x}_{1}}} \big(x_{1}, G_{2}(x_{1}) \big)\ :\ G_{2} \in \A_{2}(f) \right\}\,.
\end{equation*}
By these definitions, we have
\begin{equation*}
	\mathcal{G}^{\pm}_{1}(\widehat{x}_{2}) \subset \M_{\R}(f)
	\quad\text{and}\quad
	\mathcal{G}^{\pm}_{2}(\widehat{x}_{1}) \subset \M_{\R}(f)
\end{equation*}
and we notice that these inclusions are generally strict.
This motivates the following definition.

% ----- Reachable along different Coordinates -----
\begin{definition}[Reachability along coordinates]
\label{def:Reachability}
Let $\widehat{\bx} = (\widehat{x}_{1},\widehat{x}_{2}) \in \M_{\R}(f)$ be an arbitrary global minimizer of $f$.
We say that $\widehat{\bx}$ is \emph{reachable along the coordinate} $x_{2}$ if 
\begin{equation*}
	\widehat{\bx} \in \mathcal{G}^{+}_{1}(\widehat{x}_{2}) \cup \mathcal{G}^{-}_{1}(\widehat{x}_{2})\,,
\end{equation*}
and $\widehat{\bx}$ is \emph{reachable along the coordinate} $x_{1}$ if 
\begin{equation*}
	\widehat{\bx} \in \mathcal{G}^{+}_{2}(\widehat{x}_{1}) \cup \mathcal{G}^{-}_{2}(\widehat{x}_{1})\,.
\end{equation*}
\end{definition}
% -------------------------------------------------

\begin{remark}
In other words $\widehat{\bx}\in (\widehat{x}_{1},\widehat{x}_{2})$ is reachable along the coordinate $x_{2}$ if there is a $G_1 \in \mathcal{A}_{1}(f)$ such that 
\begin{equation*}
	\lim_{\substack{x_{2} \to \widehat{x}_{2} \\ x_{2} \neq \widehat{x}_{2}}} \big(G_{1}(x_{2}) , x_{2} \big) \in \M_{\R}(f)\,.
\end{equation*}
\end{remark}

\begin{example}
We consider the function $f_{1}$, defined in \Cref{thm:MainThm2D} (cf. also \Cref{fig:ProofThm1Plot1}). 
The set of all global minimizers of $f_{1}$ is given by $\M_{\R}(f_{1}) = \left\{ (x_{1},0) : x_{1}\in [-1,1] \right\}$.\\ 
For all $x_{2} > 0$, we have $(G_{1}(x_{2}),x_{2}) = (-1,x_{2})$ and for all $x_{2} < 0$, we have $(G_{1}(x_{2}),x_{2}) = (1,x_{2})$, so that
\begin{equation*}
	\mathcal{G}^{+}_{1}(0) = (-1,0)
	\qquad\text{and}\qquad
	\mathcal{G}^{-}_{1}(0) = (1,0)\,.
\end{equation*}
Similarly, since $(x_{1},G_{2}(x_{1})) = (x_{1},0)$ for all $x_{1}\in\RN$,
we have $\mathcal{G}^{+}_{2}(x_{1}) = \mathcal{G}^{-}_{2}(x_{1}) = (x_{1},0)$
for every $x_{1} \in [-1,1]$.

Thus all points in $\M_{\R}(f_{1})$ are reachable along the coordinate $x_{1}$
but only the points $(-1,0)$ and $(1,0)$ in $\M_{\R}(f_{1})$ are reachable along the coordinate $x_{2}$.
\end{example}

The importance of \Cref{def:Reachability} stems from the observation that if a minimizer $\widehat{\bx} \in \M_{\R}(f)$ is not reachable along a certain coordinate, then the iterative coordinate-wise algorithm will not be able to compute this minimizer.
In such a case it might happen that even though the function $f$ has global minimizers that are Turing computable, the iterative coordinate-wise algorithm may not be able to compute them because they are not reachable along a certain coordinate.

The following theorem provides such an example, namely it gives a function $f_{2}$ such that all global optimizers of $f_{2}$ that are reachable along the coordinate $x_{2}$ are not Turing computable points in $\RN^{2}$.

% ----- Theorem: Rechability 1 -----
\begin{theorem}
\label{thm:Rechability}
Let $a,b\in\RNc$ with $a,b > 0$ be arbitrary and let $\R = \R_{1} \times \R_{2}$ with $\R_{1} = [-a,a]$, $\R_{2} = [-b,b]$.
There exists a computable continuous function $f_{2} : \R \to \CN$ with the following properties
\begin{enumerate}
\item $f_{2} \in \mathcal{C}^{1}(\RN^{2})$.
\item The function $f_{2}(\cdot,x_{2})$ is strictly convex for every fixed $x_{2}\in\RN$, $x_{2}\neq 0$, and $f_{2}(\cdot,x_{2})$ is it a computable continuous function for every $x_{2}\in\RNc$.
\item The function $f_{2}(x_{1},\cdot)$ is strictly convex for every fixed $x_{1}\in\RN$, and $f_{2}(x_{1},\cdot)$ is a computable continuous function for every $x_{1}\in\RNc$.
\item $f_{2}$ has only global optimizers and the set of all global optimizers is a closed interval on the $x_{1}$-axis.
\item Every $(\widehat{x}_{1}, \widehat{x}_{2}) \in \M_{\R}(f_{2})$ that can be reached along the coordinate $x_{2}$ is not Turing computable in $\RN^{2}$.
\end{enumerate}
\end{theorem}
% ----------------------------------
As an immediate consequence of the last statement of this theorem we obtain the following negative answer to \Cref{question:effectiveConv}.
\begin{corollary} 
\label{cor:f2}
Let $f_{2} : \RN\times \RN \to \RN$ be the function defined in (the proof of) \Cref{thm:Rechability} and let $\left\{ (x_{n}, \widehat{x}_{2}) \right\}_{n\in\NN}$ be an arbitrary sequence that converges to a global minimum $(\widehat{x}_{1},\widehat{x}_{2}) \in \M_{\R}(f_{2})$ of $f_{2}$, then this convergence cannot be effective.
\end{corollary}

\begin{proof}[Proof of \Cref{thm:Rechability}]
As in the proof of \Cref{thm:MainThm2D}, we define $f_{2}$ on the entire plane $\RN^{2}$ and restrict it later to the computable rectangle $\R \subset \RN^{2}$.
To this end, let $g_{*} \in \C^{1}(\RN)$ be a function as defined in \Cref{sec:AppAuxFunc} by \eqref{equ:gStar} based on a number $\xi_{*} \in (0,1)$ that is \emph{not computable}, i.e. $\xi_{*} \notin\RNc$.
We refer to \Cref{lem:Properties_gstar} for properties of $g_{*}$ and to \Cref{fig:FunctionG} for an illustration of such a function.
Based on $g_{*}$, we define the function $f_{2} : \RN^{2} \to \RN$  by
\begin{equation}
\label{equ:f2}
	f_{2}(x_{1},x_{2}) = g_{*}(x_{1}) + u(x_{1},x_{2})
\end{equation}
with $u : \RN^{2} \to \RN$ given by
\begin{equation*}
	u(x_{1},x_{2}) = \left\{\begin{array}{lcc}
		x^{2}_{2}\, \E^{-\alpha x_{1}} & : & x_{1}\in\RN\,, x_{2}<0\\[1ex]
		x^{2}_{2}\, \E^{\alpha x_{1}} & : & x_{1}\in\RN\,, x_{2}\geq 0
	\end{array}\right.
\end{equation*}
for an arbitrary positive $\alpha \in\RN$.
\Cref{fig:ProofThm2Plot1} illustrates the shape of the so defined function $f_{2}$. 
We will now verify the properties of $f_{2}$ claimed by the theorem:

1) We show that the partial derivatives of $f_{2}$ exist and are continuous on the plane $\RN\times \RN$.
To this end we write $f^{+}_{2}$ and $f^{-}_{2}$ for $f_{2}$ restricted to the upper and lower half plane, respectively, i.e.
\begin{equation*}
\begin{array}{ll}
	f^{+}_{2}(x_{1},x_{2}) = g_{*}(x_{1}) + x^{2}_{2}\, \E^{\alpha x_{1}}\,, &(x_{1},x_{2}) \in \RN\times\RN_{+}\\[0.7ex]
	f^{-}_{2}(x_{1},x_{2}) = g_{*}(x_{1}) + x^{2}_{2}\, \E^{-\alpha x_{1}}\,, &(x_{1},x_{2}) \in \RN\times\RN_{-}
\end{array}	
\end{equation*}
Since $g_{*}$ is continuously differentiable, it is easy to see that the partial derivatives of these two functions are continuous on the corresponding half planes and we only have to verify the continuity at $x_{2} = 0$.
To this end, we observe that for any arbitrary $A\in\NN$, one has
\begin{equation*}
	\lim_{x_{2}\to 0} \sup_{x_{1} \in [-A,A]} \left| \frac{\partial f^{+}_{2}}{\partial x_1}(x_{1},x_{2}) - \frac{\partial f^{+}_{2}}{\partial x_1}(x_{1},0) \right|
	= \lim_{x_{2}\to 0} \alpha\, x^{2}_{2}\, \E^{\alpha A} = 0
\end{equation*}
and the same result holds for $\frac{\partial f^{-}_{2}}{\partial x_{1}}$, showing that $\frac{\partial f_{2}}{\partial x_{1}} \in \C(\RN)$.
In exactly the same way, one shows that $\frac{\partial f_{2}}{\partial x_{2}} \in \C(\RN)$ which finally proves that $f_{2}\in\C^{1}(\RN^{2})$.

2) For a fixed $x_{2} \neq 0$, $f_{2}(\cdot,x_{2})$ is strictly convex, because it is the sum of the convex function $g_{*}$ (cf. \Cref{lem:Properties_gstar}) and of the strictly convex exponential function.
Moreover, since $g_{*}$ and the exponential function are computable continuous functions, also its sum $f_{2}(\cdot,x_{2})$ is a computable continuous function for $x_{1}\in\RNc$.

3) For an arbitrary but fixed $x_{1} \in \RN$ the function $f_{2}(x_{1},\cdot)$, has the form of a parabola for $x_{2} \geq 0$ and $x_{2} \leq 0$ and the function is continuous on $\RN$:
\begin{equation}
\label{equ:f2parabola}
	f_{2}(x_{1}, x_{2}) = g_{*}(x_{1}) + \E^{\pm\alpha x_{1}}\cdot x^{2}_{2}\,,
	\qquad x_{2}\in\RN\,,
\end{equation}
and where the sign in the exponential function depends on the sign of $x_{1}$.
Therefore $f(x_{1},\cdot)$ is strictly convex and a computable continuous function.

4) In view of \eqref{equ:f2parabola}, we see that the unique minimum with respect to $x_{2}$ is always attained at $x_{2} = 0$.
Moreover, by the properties of $g_{*}$ (cf. \Cref{lem:Properties_gstar}), we have
\begin{equation*}
	\begin{array}{lll}
	f_{2}(x_{1}, 0) = 0 & \text{for} & x_{1} \in [-\xi_{*},\xi_{*}]\\
	f_{2}(x_{1}, 0) > 0 & \text{for} & x_{1} \notin [-\xi_{*},\xi_{*}]
	\end{array}
\end{equation*}
where $\xi_{*}\in (0,1)$ is the number used to define the function $g_{*}$ (cf. \Cref{sec:AppAuxFunc}).
So the set of all minimizes of $f_{2}$ is given by
\begin{equation*}
	\M_{\R}(f_{2}) = \left\{ (x_{1},0) \in \RN^{2} : x_{1} \in [-\xi_{*},\xi_{*}] \right\}\,,
\end{equation*}
and all of these minimizers are global minimizers.

% ----- Figure 5 -----
\begin{figure}[t]
	\centering
	\subfloat{\includegraphics[scale=0.4]{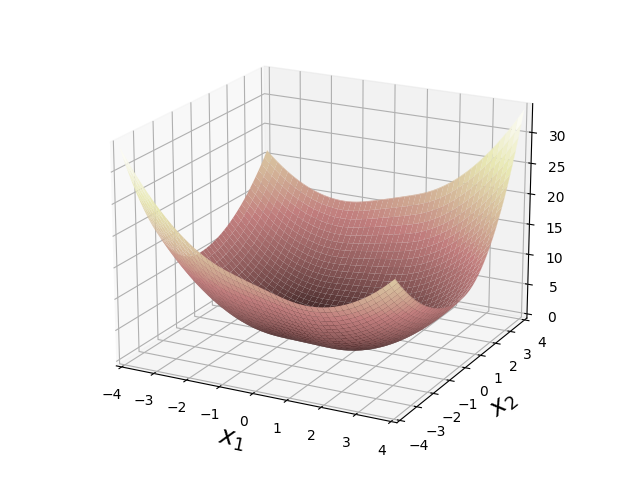}}
	\subfloat{\includegraphics[scale=0.4]{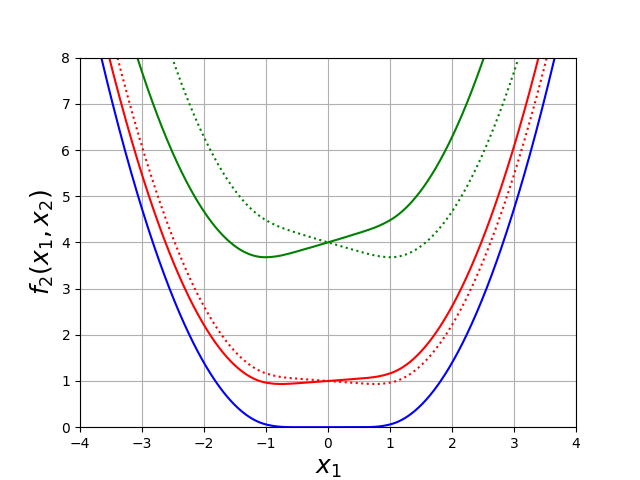}}	
	\caption{Illustration of the function $f_{2}$ defined in \eqref{equ:f2} with $\alpha = 0.1$ and using a $g_{*}$ based on $\xi_{*} = 1/2$ and with the sequence $\xi_{n} = \xi_{*} + 2^{-n}$ (cf. \Cref{sec:AppAuxFunc}).
	The plot on the right shows the function $f_{2}(x_{1},x_{2})$ (solid line) and $f_{2}(x_{1},-x_{2})$ (dotted line) for fixed $x_{2} = 0$ (blue), $x_{2} = \pm 1.0$ (red), and $x_{2} = \pm 2.0$ (green).}
	\label{fig:ProofThm2Plot1}
\end{figure}
% --------------------

5) One easily verifies that for every fixed $x_{2} \in \RN$, one has $\lim_{x_{1}\to\pm\infty} f_{2}(x_{1},x_{2}) = + \infty$.
Consequently, because $f_{2}(\cdot,x_{2})$ is strictly convex for all $x_{2}\neq 0$, the set of all local minimizers $\M_{1}(x_{2})$ (cf. \eqref{equ:set_Minl} for the definition of these sets) contains exactly one element $x_{1}$
for every $x_{2}\neq 0$.
This $x_{1}$ is the value of the assignment function $G_{1} \in \A_{1}(f_{2})$ at $x_{2}$, i.e. $x_{1} = G_{1}(x_{2})$.
So for $x_{2}\in\RNc$, $x_{2}\neq 0$ the function $f_{2}(\cdot,x_{2})$ is a continuous computable function that has a unique minimizer.
Therefore this minimizer is a computable number, i.e. $x_{1} = G_{1}(x_{2}) \in \RNc$ for every non-zero $x_{2}\in\RNc$ \cite[Chapter~$6$]{PourEl_Computability}.

Now, we fix $x_{2} > 0 $ and consider $f_{2}(x_{1},x_{2})$ for all values $x_{1} \geq -\xi_{*}$.
By the definition of $f_{2}$, one easily sees that $f_{2}(x_{1},x_{2})$ is strictly increasing for increasing $x_{1} \geq -\xi_{*}$.
Together with the observation that $\partial f_{2}/ \partial x_{1}$ is a continuous function on $\RN$, it follows that the minimizer of $f_{2}(x_{1},x_{2})$ with respect to $x_{1}$ has to be strictly smaller that $-\xi_{*}$, i.e. $G_{1}(x_{2}) < -\xi_{*}$ for all $x_{2} > 0$.
A similar argument for $x_{2} < 0$ gives finally
\begin{align}
\label{equ:G1:PlusMinus}
	\begin{array}{lll}
	G_{1}(x_{2}) > \xi_{*} > 0  & \text{for} & x_{2} < 0\\
	G_{1}(x_{2}) \in [-\xi_{*},\xi_{*}] & \text{for} & x=0\\
	G_{1}(x_{2}) < -\xi_{*} < 0 & \text{for} & x_{2} > 0
	\end{array}.
\end{align}
So we have the same situation as in the proof of \Cref{thm:MainThm2D}, i.e. we have (uncountably) many assignment functions $G_{1} \in \A_{1}(f_{2})$.
All of them are equal for $x_{2} \neq 0$ and they differ only by their value $G_{1}(0)$ which can be any number in the interval $[-\xi_{*},\xi_{*}]$.

Next we investigate the behavior of $G_{1}(x_{2})$ as $x_{2}$ converges to zero from above and from below, respectively.
First, we consider the case $x_2 > 0$ and note again that $G_{1}(x_{2})$ is the minimizer of $f_{2}(x_{1},x_{2})$ with respect to $x_{1}$ for a fixed $x_{2}$.
Thus $G_{1}(x_{2})$ is implicitly given by the equation
\begin{equation*}
	0= \frac{\partial f_{2}}{\partial x_{1}}(x_{1},x_{2})
	= g_{*}'(x_{1}) + \alpha\, x^{2}_{2}\, \E^{\alpha x_{1}}
	=: F(x_{1} , x_2)\,.
\end{equation*}
Now we choose two arbitrary points $\widetilde{x}_{2} > x_{2} > 0$ with the corresponding values $x_{1} = G_{1}(x_{2})$ and $\widetilde{x}_{1} = G_{1}(\widetilde{x}_{2})$, i.e. 
\begin{align*}
	F(x_{1},x_{2})
	&=  g'_{*}(x_{1}) + \alpha x^{2}_{2}\, \E^{\alpha x_{1}}
	= 0\\
	F(\widetilde{x}_{1},\widetilde{x}_{2})
	&= g'_{*}(\widetilde{x}_{1}) + \alpha \widetilde{x}^{2}_{2}\, \E^{\alpha \widetilde{x}_{1}}
	= 0\,.
\end{align*}
Setting the difference of these equations to zero, i.e. $F(\widetilde{x}_{1},\widetilde{x}_{2}) - F(x_{1},x_{2}) = 0$, yields
\begin{equation*}
	\frac{1}{\alpha}\big[ g'_{*}(x_{1}) - g'_{*}(\widetilde{x}_{1}) \big]
	= \widetilde{x}^{2}_{2}\, \E^{\alpha \widetilde{x}_{1}} - x^{2}_{2}\, \E^{\alpha x_{1}}\,.
\end{equation*}
Next we use a first order Taylor expansion to get a lower bound for the value $\E^{\alpha \widetilde{x}_{1}}$, i.e.
$\E^{\alpha \widetilde{x}_{1}} \geq \E^{\alpha x_{1}} + \alpha\E^{\alpha x_{1}}(\widetilde{x}_{1} - x_{1})$. Inserting this relation in the previous equation and rearranging gives
\begin{equation*}
	\widetilde{x}^{2}_{2} - x^{2}_{2}
	\leq \frac{1}{\alpha}\big[ g'_{*}(x_{1}) - g'_{*}(\widetilde{x}_{1}) \big]\, \E^{-\alpha x_{1}} + \alpha\, \widetilde{x}^{2}_{2} (x_{1} - \widetilde{x}_{1})
	\leq \big[ \tfrac{2}{\alpha} \E^{-\alpha x_{1}} + \alpha \widetilde{x}^{2}_{2}\big] \left(x_{1} - \widetilde{x}_{1} \right)\,.
\end{equation*}
The expression in the first brackets on the right hand side is always a positive number and so we see that
\begin{equation*}
	\widetilde{x}_{2} > x_{2}\quad \text{implies}\quad \widetilde{x}_{1} = G_{1}(\widetilde{x}_{2}) < G_{1}(x_{2}) = x_{1}\,,
\end{equation*}	
i.e. for any monotonically decreasing positive sequence $\big\{ x^{(k)}_{2} \big\}_{k\in\NN}$ that converges to zero, i.e. with 
\begin{equation*}
	0 < x^{(k+1)}_{2} < x^{(k)}_{2}\,,\ \text{for all}\ k\in\NN,
	\qquad\text{and}\qquad \lim_{k\to\infty} x^{(k)}_{2} = 0\,,
\end{equation*}
the sequence $G_{1}(x^{(k)}_{2})$ is monotonically increasing, and because of the last line of \eqref{equ:G1:PlusMinus}, we have that the limit
\begin{equation*}
	\lim_{x_{2}\to +0} G_{1}(x_{2}) = \widehat{x}_{1}(+0) \leq -\xi_{*}
\end{equation*}
exists. In fact the limit is equal to $-\xi_{*}$.
To see this, we consider the function $F_{1}(x_{2}) = \min_{x_{1}\in\RN} f_{2}(x_{1},x_{2})$, $x_{2}\in\RN$.
By the definition of $f_{2}$, it follows that $F_{1}$ is a continuous function on $\RN$ that satisfies $F_{1}(x_{2}) = f_{2}(G_{1}(x_{2}),x_{2})$ and with
\begin{equation}
\label{equ:limitF1}
	\lim_{x_{2}\to 0} F_{1}(x_{2}) = 0\,.
\end{equation}
Assume now that $\widehat{x}_{1}(+0) < -\xi_{*}$. Then the continuity of $F_{1}$ implies
\begin{equation*}
	\lim_{x_{2}\to +0} F_{1}(x_{2})
	= \!\!\lim_{x_{2} \to +0} f_{2}\big(G_{1}(x_{2}),x_{2}\big)
	= f_{2}\big(\widehat{x}_{1}(+0) ,x_{2}\big) > 0
\end{equation*}
which contradicts \eqref{equ:limitF1}, and so $\widehat{x}_{1}(+0) = -\xi_{*}$.
In the same way, one considers the case $x_{2} < 0$. Then one finally obtains
\begin{equation*}
	\lim_{x_{2}\to +0} G_{1}(x_{2}) = -\xi_{*}
	\qquad\text{and}\qquad
	\lim_{x_{2}\to -0} G_{1}(x_{2}) = \xi_{*}\,.
\end{equation*}
It is important to note that this holds for all (uncountably many) assignment functions $G_{1} \in \A_{1}(f_{2})$.
This shows that 
\begin{equation*}
	\mathcal{G}^{+}_{1}(0) = \left\{ (-\xi_{*},0) \right\}
	\qquad\text{and}\qquad
	\mathcal{G}^{1}_{1}(0) = \left\{ (\xi_{*},0) \right\}\,.
\end{equation*}	
i.e. each of these two sets contain only one point. So only these two points are reachable along the coordinate $x_{2}$.	
Since $\xi_{*}\notin \RNc$ this finishes the proof. %of Property~$5$.
\end{proof}

\begin{remark}
The function $f_{2}$ constructed in the previous proof is continuously differentiable. 
However, similarly as in \Cref{rem:Glattheit_f1}, we note that for every $K\in\NN$, 
one can constructs a function $f_{2}\in\mathcal{C}^{K}(\RN)$ such that all partial derivatives of $f_{2}$ up to order $K$ are computable continuous functions 
and that also satisfies Properties 2) - 5) of \Cref{thm:Rechability}.
To this end, one only has to replace the function $g_{*}$, given in \Cref{sec:AppAuxFunc} by a function that is piecewise a polynomial of order $K+1$ and $K$-times continuously differentiable.
\end{remark}

% =============================================
% ========== Outlook ==========================
% =============================================
\section{Extensions}
\label{sec:Extension}

The functions $f_{1}, f_{2} : \RN^{2} \to \RN$ constructed in \Cref{thm:MainThm2D,thm:Rechability} are fairly simply. Since these functions are defined on $\RN^{2}$, the corresponding block coordinate optimization method \eqref{equ:ItarativLocalOpt} is automatically a coordinate-wise iterative optimization as given in Algorithm~\ref{alg:CoordinateWise2}.
However, the derived results hold also for general block coordinate optimization methods as discussed at the beginning of \Cref{sec:IterativeMethod} with at least two blocks.

The main results of this paper (\Cref{thm:MainThm2D,thm:Rechability}) are formulated for constraint minimization problems of the form \eqref{equ:MinProbGen}.
However, it becomes clear from the constructions of $f_{1}$ and $f_{2}$ in the proofs of these two theorems, that the statements hold also for unconstrained optimization problems.
Indeed, for the proofs of these statements, it was merely important that the functions only have global minima (i.e. the functions have no local minima that are not also global minima) and that all these global minima lie in a bounded region of $\RN^{m}$.

\section{Summary and discussion}
\label{sec:Summary}

To solve the optimization problem \eqref{equ:MinProbGen} on a digital computer, one aims for an algorithm with two inputs: the function $f$ and an integer $M\in\NN$. 
Based on these two inputs, the algorithm should compute $\Min_{\R}(f)$ and a corresponding  minimizer $\widehat{\bx}$ with a guaranteed error of at most $2^{-M}$. 
A common way to derive such an algorithm is based on the application of an iterative optimization strategy that optimizes successively over the single coordinates (cf. Algorithm~\ref{alg:CoordinateWise2}).

We have shown in this paper that even for a fixed given function $f_{*}$, it might be impossible to find an effective implementation of such an iterative optimization algorithm.
Since it is impossible to construct a specific algorithm for the fixed function $f_{*}$, it is \emph{a fortiori} impossible to construct a general algorithm who takes $f$ as an input and which is able solve the optimization problem for a larger set (including $f_{*}$) of functions.
The paper discussed two reasons why such an effective implementation of the iterative optimization procedure may fail.
\begin{itemize}
\item[1)] The $\arg\min$-operation \eqref{equ:ItarativLocalOpt} inside the iteration might not be Turing computable for the given function $f_{*}$.
\item[2)] The iterative algorithm converges for any arbitrary initialization vector to a non-computable minimizer of $f_{*}$.
\end{itemize}
\Cref{sec:ProblemWithAssignmentFunc,sec:Reachability} provided two concrete functions $f_{1}$ and $f_{2}$ such that a behavior according to 1) and 2) occur, respectively.

Behavior~1) was illustrated by an example in \Cref{thm:MainThm2D}.
It was shown that the iterative algorithm converges (after at most two steps) to one of the computable global optimizers $(1,0)$ or $(-1,0)$ of the function $f_{1} : \RN^{2}\to\RN$.
However, for any arbitrary initialization vector, already the first $\arg\min$-operation is not Turing computable which means that this step cannot be computed effectively on a Turing machine.
Even more, not only can the $\arg\min$-operation generally not be solved algorithmically, it is even not possible to have a non-trivial approximation of this iterative optimization step (cf. \Cref{sec:ApproximationAssFkt}).

Behavior 2) was illustrated by a function $f_{2}$ in \Cref{thm:Rechability}.
For this function, the iterative algorithm generates a sequence $\left\{ \widetilde{\bx}^{(k)} \right\}_{k\in\NN}$ of local minimizers that always converges either to the global minimizer $(-\xi_{*},0)$ or to the global minimizer $(\xi_{*},0)$, even thought there exist uncountably many other global minimizers.
Then, if $\xi_{*}$ is not a computable number, the convergence towards these two minimizers cannot be effective.

We would like to emphasize, that the shown negative property is a consequence of the local (coordinate-wise) optimization strategy.
Indeed, the two functions $f_{1}$ and $f_{2}$ constructed in \Cref{sec:ProblemWithAssignmentFunc,sec:Reachability}, respectively, both have at least one global minimizer $\widehat{\bx} \in \RN^{m}$ that is computable.
Therefore it is always possible to find a computable sequence $\left\{ \widehat{\bx}^{(k)} \right\}_{k\in\NN} \subset \RN^{m}$ that effectively converges to $\widehat{\bx}$.
Indeed, since $\widehat{\bx}$ is computable, each component $\widehat{x}_{n}$, $n=1,2,\dots,m$ of $\widehat{\bx}$ is a computable number and so there exists a Turing machine $\TM_{n}$ that computes for input $M\in\NN$ a rational number $\widetilde{x}_{n}(M) = \TM_{n}(M)$ such that
$\left| \widehat{x}_{n} - \widetilde{x}_{n}(M) \right| < 2^{-M}$ for every $M\in\NN$.
Then it follows that the sequence $\left\{ \widetilde{\bx}(M)\right\}_{M\in\NN}$ with $\widetilde{\bx}(M) = (\widetilde{x}_{1}(M), \dots, \widetilde{x}_{m}(M))$ satisfies
\begin{equation*}
	\left\| \widehat{\bx} - \widetilde{\bx}(M) \right\|_{2} < \sqrt{m}\, 2^{-M}\,,
	\quad\text{for all}\ M\in\NN\,.
\end{equation*}
Thus $\left\{ \widetilde{\bx}(M)\right\}_{M\in\NN}$ converges effectively to the global optimizer $\widehat{\bx}$, 
and so we have found an effective numerical procedure to compute the global optimizer $\widehat{\bx}$.
So for both functions, there exists a numerical algorithm for computing the global minimizer of \eqref{equ:MinProbGen}.
However, such an algorithm is based on a global optimization strategy. Iterative algorithms that apply a local (i.e. a coordinate-wise) optimization strategy cannot effectively solve \eqref{equ:MinProbGen}.

Our results are also relevant for cases where a global optimization is hard to implement or where it is even impossible to implement such a global optimization (cf., e.g., \cite{PalomarEldar_ConvexOpt,liu2024survey}).
As an example consider a decentralized communication system where it is impossible to collect global information about the whole network at a central point, but where the optimization has to performed locally, based on only restricted knowledge on the network.
Our results are relevant for these scenarios since they correspond to a local (decentralized) optimization, discussed in this paper.

We would also like to point out that our findings are in sharp contrast to the common heuristic.
By this heuristic, it is assumed that a global, jointly optimization over all variables (i.e. over all degrees of freedoms) is too complex, but it is less complex to optimize locally over some (or even a single) dimensions while keeping the other coordinates fixed and to iterate over all dimensions.
However, as discussed in this paper, the local optimization strategy needs to solve iteratively the $\arg\min$-operation dependent on some parameters (the fixed coordinates).
We have shown that the parameter dependent minimizers are uniquely determined, but even for very simple functions in two variables, the mapping from the parameters to the unique minimizer is generally not Turing computable (cf. \Cref{thm:MainThm2D}).
Thus the unique minimizer can generally not be computed effectively from the parameters.
In other words, the computation of the local minimizers of the iterative algorithm is even too complex for a digital computer.
Nevertheless, the function $f_{1}$ from \Cref{thm:MainThm2D} has very simple properties that allows one to determine a global minimizer by standard joint optimization.
So this result is contrary to the common heuristic that an iterative local optimization is less complex than a joint global optimization.

% =============================================
% ========== Appendix =========================
% =============================================
\appendix\section{An auxiliary function}
\label{sec:AppAuxFunc}

In the proof of \Cref{thm:Rechability} we need a particular auxiliary function $g_{*} : \RN\to\RN$.
This appendix defines $g_{*}$ and proves properties of $g_{*}$ needed in the proof of \Cref{thm:Rechability}.
To this end, we first choose an arbitrary number $\xi_{*} \in (0,1)$ that is the limit of a computable sequence $\bsxi = \left\{ \xi_{n} \right\}_{n\in\NN} \subset\QN$ of strictly monotonically decreasing rational numbers, i.e.
\begin{equation}
\label{equ:Seq_a}
	\xi_{n+1} < \xi_{n}\,, \ \text{for all}\ n\in\NN\,,
	\qquad\text{and}\qquad
	\lim_{n\to\infty} \xi_{n} = \xi_{*}\,.
\end{equation}
\begin{remark}
It is important to notice that we require only that $\xi_{*}$ is the limit of a monotonic decreasing sequence. So there might exist no monotonically increasing sequence that converges to $\xi_{*}$.
In such a case $\xi_{*}$ would not be a computable number. So in other words, we explicitly allow $\xi_{*}$ to be non-computable.
\end{remark}
For every $n\in\NN$, we define a continuously differentiable function $g_{n} \in \C^{1}(\RN)$ by
\begin{equation*}
	g_{n}(x) = \left\{\begin{array}{ccc}
	(x + \xi_{n})^{2} & : & x < -\xi_{n}\\
	0 & : & -\xi_{n} \leq x \leq \xi_{n}\\
	(x - \xi_{n})^{2} & : & x > \xi_{n}
	\end{array}\right. .
\end{equation*}
Therewith, we define finally the function
\begin{equation}
\label{equ:gStar}
	g_{*}(x) = \textstyle\sum^{\infty}_{n=1} 2^{-n}\, g_{n}(x)\,,
	\quad x\in\RN\,.
\end{equation}
The function $g_{*}$ depends on the chosen $\xi_{*}$ and on the chosen sequence $\bsxi$.
However, the important properties of $g_{*}$ that are needed in this paper are independent of the choice of $\xi_{*}$ and $\bsxi$.
\Cref{fig:FunctionG} illustrates the shape of $g_{*}$ for some values of $\xi_{*}$, and
the next lemma collects the properties of $g_{*}$ that are needed in the proof \Cref{thm:Rechability}.

% ----- Figure Function g_star -----
\begin{figure}[t]
	\centering
	\includegraphics[scale=0.4]{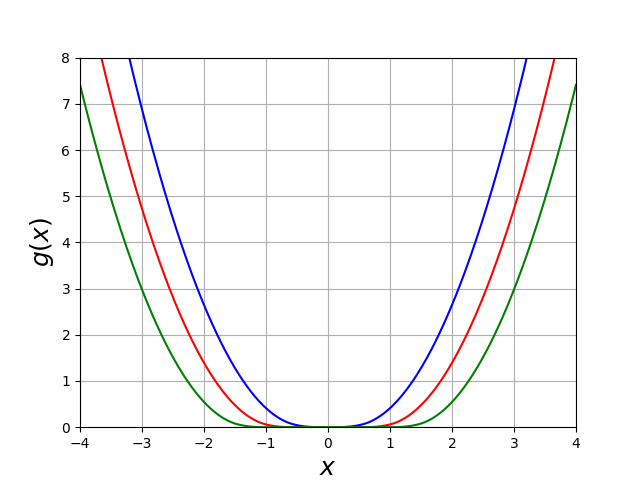}
	\caption{Illustration of the function $g_{*}$ defined in \eqref{equ:gStar} for $\xi_{*} = 1/20$ (blue), $\xi_{*} = 1/2$ (red), and $\xi_{*} = 0.95$ (green).
	In all three cases, the sequence $\xi_{n} = \xi_{*} + 2^{-n}$, $n\in\NN$ was used to produce these plots.}
	\label{fig:FunctionG}
\end{figure}
% ----------------------------------

\begin{lemma}
\label{lem:Properties_gstar}
Let $\xi_{*} \in (0,1)$ be arbitrary and let $\bsxi = \left\{ \xi_{n} \right\}_{n\in\NN} \subset\QN$ be a sequence that satisfies \eqref{equ:Seq_a}.
Then the function $g_{*} : \RN \to \RN$ defined in \eqref{equ:gStar} has the following properties:
\begin{enumerate}
\item $g_{*}$ is even on $\RN$, i.e. $g_{*}(-x) = g_{*}(x)$ for all $x\in\RN$.
\item $g_{*}(x) \geq 0$ for all $x\in\RN$.
\item $g_{*}(x) = 0$ for all $x\in [-\xi_{*},\xi_{*}]$.
\item $g_{*}$ is continuously differentiable on $\RN$.
\item The first derivative $g'_{*}$ is piecewise linear and $g'_{*}(x_{1}) - g'_{*}(\widetilde{x}_{1}) \leq 2 \left(x_{1} - \widetilde{x}_{1} \right)$ for all $x_{1}, \widetilde{x}_{1} \in \RN$.
%\begin{equation}
%\label{equ:gStichstarDiff}
	%g'_{*}(x_{1}) - g'_{*}(\widetilde{x}_{1}) \leq 2 \left(x_{1} - \widetilde{x}_{1} \right)\,,
	%\quad\text{for all}\ x_{1}, \widetilde{x}_{1} \in \RN\,.
%\end{equation}
\item $g_{*}$ is convex.
\item $g_{*}$ is a computable continuous function on $\RN$.
\end{enumerate}
\end{lemma}

\begin{proof}
Properties 1 -- 3 are obvious from the definition.
Property 4 follows because the series \eqref{equ:gStar} that defines $g_{*}$ as well as the corresponding series for $g'_{*}$ are absolute convergent.

To verify Property 5, we define for $n=1, 2, 3, \dots$, the intervals
\begin{equation*}
\begin{array}{lcl}
	\mathcal{I}_{-1} = (-\infty,-\xi_{1}) & \text{and} & \mathcal{I}_{1} = (\xi_{1},\infty) \\[0.7ex]
	\mathcal{I}_{-n} = [-\xi_{n-1}, -\xi_{n}) & \text{and} & \mathcal{I}_{n} = (\xi_{n},\xi_{n-1}]\\[0.7ex]
	\mathcal{I}_{\infty} = [-\xi_{*},\xi_{*}]\,.
\end{array}	
\end{equation*}
Then \eqref{equ:gStar} can be written as
\begin{equation*}
	g_{*}(x) = \sum^{\infty}_{k=n} \frac{1}{2^{k}} \left( x - \xi_{k} \right)^{2}\,,
	\quad \text{for}\ x\in \mathcal{I}_{n}
\end{equation*}
and $n=1,2,\dots$, and similarly for $x\in\mathcal{I}_{-n}$.
Differentiating term by term within the sum yields 
\begin{equation}
\label{equ:gStrichStar}
	g_{*}'(x) =
	\left\{\begin{array}{lll}
	c_{1}(n)\, x + c_{0}(n) & : & x \in \mathcal{I}_{-n}\\
	0 & : & x \in \mathcal{I}_{\infty}\\
	c_{1}(n)\, x - c_{0}(n) & : & x \in \mathcal{I}_{n}
	\end{array}\right. ,	
\end{equation}
with positive constants $c_{0}(n)$ and $c_{1}(n)$ that satisfy
\begin{align}
\label{equ:slop_gStichStar}
	0 < &c_{1}(n) = \frac{1}{2^{n-2}} \leq 2
	\qquad\text{and}\\\
	c_{1}(n)\, \xi_{*} \leq &c_{0}(n) = \sum^{\infty}_{k=n}\frac{\xi_{k}}{2^{k-1}} \leq c_{1}(n)\, \xi_{n}\,.\nonumber
\end{align}
This shows that $g'_{*}$ is linear on all intervals $\mathcal{I}_{n}$ and one can easily verify (directly) that $g'_{*}$ is continuous on $\RN$.
Property 5 follows from the observation that the slopes of the linear pieces are given by $c_{1}(n)$ which satisfy \eqref{equ:slop_gStichStar}.

The convexity, i.e., Property 6, follows from \eqref{equ:gStrichStar}. It shows that $g'_{*}$ is continuous and monotonically increasing because the sequence $\left\{ c_{1}(n) \right\}_{n\in\NN}$ is monotonically decreasing.
The computability of $g_{*}$ follows from the observation that $g_{*}$ is a polynomial on every interval $\mathcal{I}_{n}$ and that the boundary points $\xi_{n}$ of these intervals are rational numbers.
\end{proof}

% ============================================
% ========== Bibliography =====================
% =============================================
%\bibliographystyle{siamplain}
%\bibliography{../../pub_Books,../../pub_Pohl,../../publications}

\begin{thebibliography}{10}

\bibitem{Arimoto_TIT72}
{\sc S.~Arimoto}, {\em {An algorithm for computing the capacity of arbitrary
  discrete memoryless channels}}, {IEEE Trans. Inf. Theory}, 18 (1972),
  pp.~14--20.

\bibitem{AvigadBrattka_2014}
{\sc J.~Avigad and V.~Brattka}, {\em {Computability and analysis: The legacy of
  Alan Turing}}, in {Turing's Legacy: Developments from Turing's Ideas in
  Logic}, {Lecture Notes in Logic, Bd. 42}, {Cambridge University Press}, New
  York, 2014, pp.~1--47.

\bibitem{Beck_SIAMJOpt13}
{\sc A.~Beck and L.~Tetruashvili}, {\em {On the convergence of block coordinate
  descent type methods}}, SIAM J. Optim., 23 (2013), pp.~2037--2060.

\bibitem{BenTalNemirovski_SIAM01}
{\sc A.~Ben-Tal and A.~Nemirovski}, {\em {Lectures on Modern Convex
  Optimization}}, {MPS-SIAM Series on Optimization}, Society for Industrial and
  Applied Mathematics (SIAM), Philadelphia, PA, USA, 2001.

\bibitem{BenidisFengPalomar_Now18}
{\sc K.~Benidis, Y.~Feng, and D.~P. Palomar}, {\em {Optimization methods for
  financial index tracking: From theory to practice}}, {Foundations and Trends
  in Signal Processing}, 3 (2018), pp.~171--279.

\bibitem{Bertsekas_ConvexOpt}
{\sc D.~P. Bertsekas}, {\em {Convex Optimization Algorithms}}, Athena
  Scientific, Nashua, USA, 2015.

\bibitem{Blahut_TIT72}
{\sc R.~E. Blahut}, {\em {Computation of channel capacity and rate-distortion
  functions}}, {IEEE Trans. Inf. Theory}, 18 (1972), pp.~460--473.

\bibitem{BSP_TIT23}
{\sc H.~Boche, R.~F. Schaefer, and H.~V. Poor}, {\em {Algorithmic computability
  and approximability of capacity-achieving input distributions}}, {IEEE Trans.
  Inf. Theory}, 69 (2023), pp.~5449--5462.

\bibitem{Boyd_ConvexOpt}
{\sc S.~Boyd and L.~Vandenberghe}, {\em {Convex Optimization}}, Cambridge
  University Press, Cambridge, 20004.

\bibitem{CandesRecht_MatrixCompl09}
{\sc E.~J. Cand{\`e}s and B.~Recht}, {\em {Exact matrix completion via convex
  optimization}}, {Found. Comput. Math.}, 9 (2009), pp.~717--772.

\bibitem{CandesTao_IT05}
{\sc E.~J. Cand{\`e}s and T.~Tao}, {\em {Decoding by linear programming}},
  {IEEE Trans. Inf. Theory}, 51 (2005), pp.~4203--4215.

\bibitem{ComputingAsDisciplin}
{\sc D.~E. Comer, D.~Gries, M.~C. Mulder, A.~Tucker, A.~J. Turner, P.~R. Young,
  and P.~J. Denning}, {\em {Computing as a discipline}}, {Commun. ACM}, 32
  (1989), pp.~9--23.

\bibitem{Csizar_TIT74}
{\sc I.~Csisz{\'a}r}, {\em On the computation of rate-distortion functions},
  IEEE Trans. Inf. Theory, 20 (1974), pp.~122--124.

\bibitem{DeSantis_SIAMJOpt16}
{\sc M.~De~Santis, S.~Lucidi, and F.~Rinaldi}, {\em {A fast active set block
  coordinate descent algorithm for $\ell_{1}$-regularized least squares}}, SIAM
  J. Optim., 26 (2016), pp.~781--809.

\bibitem{Friedman_CompComplexity_84}
{\sc H.~M. Friedman}, {\em {The computational complexity of maximization and
  integration}}, Adv. Math., 53 (1984), pp.~80--98.

\bibitem{GrippoGaussSeidel2000}
{\sc L.~Grippo and M.~Sciandrone}, {\em {On the convergence of the block
  nonlinear Gauss–Seidel method under convex constraints}}, Oper. Res. Lett.,
  26 (2000), pp.~127--136.

\bibitem{Ko_Complexity91}
{\sc K.-I. Ko}, {\em {Complexity Theory of Real Functions}}, Birkh{\"a}user,
  Basel, 1991.

\bibitem{LeeBK_TIT24}
{\sc Y.~Lee, H.~Boche, and G.~Kutyniok}, {\em {Computability of optimizers}},
  {IEEE Trans. Inf. Theory}, 70 (2024), pp.~2967--2983.

\bibitem{LeeBoKu_CompOpt_24}
{\sc Y.~Lee, H.~Boche, and G.~Kutyniok}, {\em {Computability of optimizers for
  AI and data science}}, in {Handbook of Numerical Analysis}, Elsevier B.V.,
  New York, 2024, pp.~1--54.

\bibitem{liu2024survey}
{\sc Y.-F. Liu, T.-H. Chang, M.~Hong, Z.~Wu, A.~M.-C. So, E.~A. Jorswieck, and
  W.~Yu}, {\em A survey of recent advances in optimization methods for wireless
  communications}, pre-print,  (2024).
\newblock arXiv:2401.12025.

\bibitem{Nesterov_SIAMOptm_12}
{\sc Y.~Nesterov}, {\em {Efficiency of coordinate descent methods on huge-scale
  optimization problems}}, {SIAM J. Optim.}, 22 (2012), pp.~341--362.

\bibitem{PalomarEldar_ConvexOpt}
{\sc D.~P. Palomar and Y.~C. Eldar}, eds., {\em {Convex Optimization in Signal
  Processing and Communications}}, Cambridge University Press, Cambridge, UK,
  2009.

\bibitem{PourEl_Computability}
{\sc M.~B. Pour-El and J.~I. Richards}, {\em {Computability in Analysis and
  Physics}}, Springer-Verlag, Berlin, 1989.

\bibitem{Powell_MathProgramm1973}
{\sc M.~J.~D. Powell}, {\em {On search directions for minimization
  algorithms}}, Math. Program., 4 (1973), pp.~193--201.

\bibitem{Razaviyayn_SIAMJopt13}
{\sc M.~Razaviyayn, M.~Hong, and Z.-Q. Luo}, {\em {A unified convergence
  analysis of block successive minimization methods for nonsmooth
  optimization}}, SIAM J. Optim., 23 (2013), pp.~1126--1153.

\bibitem{Specker_SatzVomMaximum}
{\sc E.~Specker}, {\em {Der Satz vom Maximum in der rekursiven Analysis}}, in
  {Ernst Specker Selecta}, G.~J{\"a}ger, H.~L{\"a}uchli, B.~Scarpellini, and
  V.~Strassen, eds., Birkh{\"a}user, Basel, 1990, pp.~148--159.

\bibitem{Turing_1937}
{\sc A.~M. Turing}, {\em {On computable numbers, with an application to the
  Entscheidungsproblem}}, {Proc. London Math. Soc.}, s2-42 (1937),
  pp.~230--265.

\bibitem{Turing_1938}
{\sc A.~M. Turing}, {\em {On computable numbers, with an application to the
  Entscheidungsproblem. A correction}}, {Proc. London Math. Soc.}, s2-43
  (1938), pp.~544--546.

\bibitem{Weihrauch_ComputableAnalysis}
{\sc K.~Weihrauch}, {\em {Computable Analysis}}, Springer-Verlag, Berlin, 2000.

\bibitem{Wright_MathProg15}
{\sc S.~J. Wright}, {\em {Coordinate descent algorithms}}, Math. Program., 151
  (2015), pp.~3--34.

\bibitem{Xu_SIALJOpt}
{\sc Y.~Xu}, {\em {Hybrid Jacobian and Gauss--Seidel proximal block coordinate
  update methods for linearly constrained convex programming}}, SIAM J. Optim.,
  28 (2018), pp.~646--670.

\bibitem{XuYin_SIAM_JIS13}
{\sc Y.~Xu and W.~Yin}, {\em {A block coordinate descent method for regularized
  multiconvex optimization with applications to nonnegative tensor
  factorization and completion}}, SIAM J. Imaging Sci., 6 (2013),
  pp.~1758--1789.

\end{thebibliography}

\end{document}